\documentclass[12pt]{amsart}
\usepackage{amsmath,amsthm,amssymb}
\usepackage[usenames,dvipsnames,svgnames,table]{xcolor}
\usepackage{algorithm2e}
\usepackage{mathdots}
\usepackage{graphicx}

\newtheorem{theorem}{Theorem}[section]
\newtheorem{corollary}[theorem]{Corollary}

\newtheorem{condition}[theorem]{Condition}
\newtheorem{lemma}[theorem]{Lemma}
\newtheorem{proposition}[theorem]{Proposition}
\newtheorem{claim}[theorem]{Claim}

\newtheorem{definition}[theorem]{Definition}
\let\olddefinition\definition
\renewcommand{\definition}{\olddefinition\normalfont}
\newtheorem{remark}[theorem]{Remark}
\let\oldremark\remark
\renewcommand{\remark}{\oldremark\normalfont}

\renewcommand\ge\geqslant
\renewcommand\geq\geqslant
\renewcommand\le\leqslant
\renewcommand\leq\leqslant

\newcommand{\ZZ}{\mathbb{Z}}
\newcommand{\NN}{{\ensuremath{\mathbb{N}}}}
\newcommand{\RR}{{\ensuremath{\mathbb{R}}}}

\newcommand{\ma}{{\ensuremath{\mathfrak{a}}}} 
\newcommand{\mb}{{\ensuremath{\mathfrak{b}}}}
\newcommand{\md}{{\ensuremath{\mathfrak{d}}}}

\newcommand{\ms}{{\ensuremath{\mathfrak{s}}}}
\newcommand{\cb}{{\ensuremath{\mathbf{c}}}}
\newcommand{\bg}{{\ensuremath{\mathbf{g}}}} 
\newcommand{\bh}{{\ensuremath{\mathbf{h}}}} 
\newcommand{\bq}{{\ensuremath{\mathbf{q}}}}
 
\newcommand{\bs}{{\ensuremath{\mathbf{s}}}} 
\newcommand{\cS}{{\ensuremath{\mathcal{S}}}}
\newcommand{\cM}{{\ensuremath{\mathcal{M}}}}
\newcommand{\cN}{{\ensuremath{\mathcal{N}}}}
\newcommand{\cO}{{\ensuremath{\mathcal{O}}}}
\newcommand{\cR}{{\ensuremath{\mathcal{R}}}}

\newcommand{\cK}{{\ensuremath{\mathcal{K}}}}

\newcommand{\cD}{{\ensuremath{\mathcal{D}}}}
\newcommand{\cE}{{\ensuremath{\mathcal{E}}}} 
\newcommand{\cF}{{\ensuremath{\mathcal{F}}}}
\newcommand{\cH}{{\ensuremath{\mathcal{H}}}} 

\newcommand{\cV}{{\ensuremath{\mathcal{V}}}}

\newcommand{\Syz}{\mathrm{Syz}}

\def\ub{\mathbf{u}}

\def\cross{\mathfrak{c}}
\newcommand{\fc}[1]{\mbox{\em #1}}
\newcommand{\ead}[1]{\texttt{\em #1}}
\def\sep{\!\!; }
\newenvironment{keyword}{\textbf{Keywords:}}{}

\title[]{Dimension and bases for geometrically continuous splines on surfaces of arbitrary topology}

\author[B. Mourrain]{Bernard Mourrain}
\author[R. Vidunas]{Raimundas Vidunas}
\author[N. Villamizar]{Nelly Villamizar}

\address{Bernard Mourrain , Inria Sophia Antipolis M\'editerran\'ee, Sophia Antipolis, France, \ead{Bernard.Mourrain@inria.fr}}
\address{Raimundas Vidunas , University of Tokyo, Tokyo, Japan, \newline\ead{rvidunas@gmail.com}}
\address{Nelly Villamizar , RICAM, Austrian Academy of Sciences, Linz, Austria, \newline\ead{nelly.villamizar@oeaw.ac.at}}

\begin{document}
\maketitle

\begin{abstract}
  We analyze the space of geometrically continuous piecewise
  polynomial functions, or splines, for rectangular and triangular patches 
  with arbitrary
  topology and general rational transition maps.  To define these spaces
  of $G^{1}$ spline functions, we introduce the concept of topological
  surface with gluing data attached to the edges shared by faces. The
  framework does not require manifold constructions and is general
  enough to allow non-orientable surfaces.  We describe compatibility
  conditions on the transition maps so that the space of
  differentiable functions is ample and show that these conditions are
  necessary and sufficient to construct ample spline spaces.  We
  determine the dimension of the space of $G^{1}$ spline functions
  which are of degree $\le k$ on triangular pieces and of bi-degree
  $\le (k,k)$ on rectangular pieces, for $k$ big enough.  A
  separability property on the edges is involved to obtain the
  dimension formula.  An explicit construction of basis functions
  attached respectively to vertices, edges and faces is proposed; 
  examples of bases of $G^{1}$ splines of small degree for topological
  surfaces with boundary and without boundary are detailed.
\end{abstract}



\section{Introduction}

The accurate and efficient representation of shapes is a major
challenge in geometric modeling.
To achieve high order accuracy in the representation of curves,
surfaces or functions, piecewise polynomials models are usually
employed.
Parametric models with prescribed regularity properties are nowadays  
commonly used in Computer Aided Geometric Design  (CAGD) to 
address these problems.
They involve so-called spline functions, which are piecewise
polynomial functions on intervals of $\RR$ with continuity
and differentiability constraints at some nodes.
Extensions of these functions to higher dimension 
are usually done by taking tensor product spline basis functions.
Curves, surfaces or volumes are represented as the image of parametric
functions expressed in terms of spline basis functions.
For instance,  surface patches are described as the image of a
piecewise polynomial (or rational) map from a rectangular domain of
$\RR^2$ to $\RR^n$.
But to represent objects with complex topology, such maps on 
rectangular parameter domains are not sufficient.
One solution which is commonly used in Computer-Aided Design (CAD) 
is to trim the B-spline rectangular patches and to ``stitch'' 
together the trimmed pieces to create the complete shape representation.
This results in complex models, which are not simple to use and to
modify, since structural rigidity conditions cannot easily be imposed along
the trimming curve between two trimmed patches.

To allow flexibility in the representation of shapes with complex
topology, another technique called geometric continuity has been studied. 
Rectangular parametric surface patches are glued along their common boundary, 
with continuity constraints on the tangent planes (or on higher osculating
spaces). In this way, smooth surfaces can be generated from
quadrilateral meshes by gluing several simple parametric
surfaces, forming surfaces with the expected smoothness property along
the edges.

This approach builds on the theory on differential manifolds, in works
such as \cite{derose}, \cite{Hahn87}, \cite{Gregory89}.
The idea of using transition maps or reparameterizations in connection with building
smooth surfaces had been used for instance by DeRose \cite{derose} in
CAGD, who gave one of the first general definitions of splines based
on fixing a parametrization.

Since these initial developments, several works focused on the construction of such $G^{1}$ surfaces
\cite{Peters:1994},
\cite{Loop:1994},
\cite{Reif:1995},
\cite{Prautzsch:1997},
\cite{Catmull:1998},
\cite{Ying:2004},
\cite{Gu:2005},
\cite{He:2006},
\cite{Gu:2008},
\cite{Vecchia:2008},
\cite{Tosun:2011},
\cite{flex},
\ldots
with polynomial, piecewise polynomial, rational or special functions
and on their use in geometric modeling applications such as surface
fitting or surface reconstruction
\cite{Eck:1996},
\cite{Shi:2004},
\cite{Lin:2007},
\ldots

The problem of investigating the minimal degree of polynomial pieces has
also been considered \cite{Peters:2002}. Other research investigates
the construction of adapted rational transition maps for a given
topological structure \cite{beccari}.  We refer to
\cite{PetersHandbook} for a review of these constructions.
Constraints that the transition maps must satisfy in order to define
regular spline spaces have also been identified \cite{Peters2010}.
But it has not yet been proved that these constraints are sufficient
for the constructions.

The use of $G^{k}$ spline functions to approximate functions over
computational domains with arbitrary topology received recently a new
attention for applications in isogeometric analysis.  In this
context, describing the space of functions, its dimension and adapted
bases is of particular importance.
A family of bi-cubic spline functions was recently introduced by Wu et
al \cite{meng} for isogeometric applications, where constant
transition maps are used, which induce singular spline basis functions at
extraordinary vertices.
Multi-patch representations of computational domains are also used in 
\cite{adaptively}, with constant transition maps at the shared edges 
of rectangular faces, using an identification of Locally Refined spline 
basis functions. 
In \cite{kapl}, $G^{k}$ continuous splines are described and the $G^{1}$
condition is transformed into a linear system of relations between
the control coefficients. The case of two rectangular patches, which
share an edge is analyzed experimentally.
In \cite{Bercovier}, the space of $G^{1}$ splines of bi-degree $\ge 4$
for rectangular decompositions of planar domains is analyzed.
Minimal Determining
Sets of points are studied, providing dimension formulae and dual basis
for $G^{1}$ spline functions over planar rectangular meshes with
linear gluing transition maps. 

Our objective is to analyze the space of $G^{1}$ spline functions
for rectangular and triangular patches with arbitrary topology
and general rational transition maps. We are interested in determining the
dimension of the space of $G^{1}$ spline functions which are of degree $\le k$ on
triangular pieces and of bi-degree $\le (k,k)$ on rectangular pieces.
To define the space  of $G^{1}$ spline functions, we introduce the
concept of topological surface with gluing data attached to the edges
shared by the faces. The framework does not require manifold constructions
and is general enough to allow non-orientable surfaces.
We describe compatibility conditions on the transition maps so that
the space of differentiable functions is ample and show that these conditions are necessary
and sufficient to construct ample spline spaces.
A separability property is involved to obtain a dimension formula of
the $G^{1}$ spline spaces of degree $\le k$ on such topological
surfaces, for $k$ big enough.
This leads to an explicit construction of basis functions attached
respectively to vertices, edges and faces.

For the presentation of these results, we structure the paper as
follows. The next section introduces the notion of topological surface
$\cM$, differentiable functions on $\cM$ and constraints on the
transition maps to have an ample space of differentiable functions.
Section 3 deals with the space of spline functions which are piecewise
polynomial and differentiable on $\cM$.  Section 4 analyzes the gluing
conditions along an edge.  Section 5 analyzes the gluing condition
around a vertex.  In Section 6, we give the dimension formula for the
space of spline functions of degree $\le k$ over a topological surface
$\cM$ and describe explicit basis constructions.  Finally, in Section
7, we detail an example with boundary edges and another one with no
boundary edges.  We also provide an appendix with an algorithmic
description of the basis construction.

\section{Differentiable functions on a topological surface}

Typically in CAGD, parametric patches are glued into surfaces by splines
(i.e., polynomial maps) from polygons in $\RR^2$. 
The simplest $C^r$ construction is with the polygons in $\RR^2$
situated next to each other, so that $C^r$ continuity across patch edges
comes from $C^r$ continuity of the coordinate functions across the polygon edges.
This is called {\em parametric continuity}. A more general construction to generate 
a $C^r$ surface from polygonal patches is called {\em geometric continuity} 
\cite{derose}, \cite{PetersHandbook}.
Inspired by differential geometry, attempts have been made 
\cite{Hahn87}, \cite{raimundas}, \cite{VidunasAMS} to define 
{\em geometrically continuous $G^r$ surfaces} from a collection of polygons in $\RR^2$
with additional data to glue their edges and differentiations.
They are defined by parametrization maps from the polygons to $\RR^3$
satisfying geometric regularity conditions along edges.

It is easy to define a $C^0$ surface from a collection of polygons and
homeomorphisms between their edges.
\begin{definition} \label{def:g0complex}
	Given a collection $\cM_2$ of (possibly coinciding) polygons $\sigma_i$ in $\RR^2$, 
	a {\em topological surface} $\cM$ is defined by giving a set of homeomorphisms 
	$\mu:\tau_{i}\to{\tau}_{j}$ between pairs of polygonal edges 
	$\tau_{i}\subset\sigma_{i}$, ${\tau}_{j}\subset{\sigma}_{j}$
	($\sigma_{i},{\sigma}_{j}\in\cM_2$). Each polygonal edge can be paired 
	with at most one other edge, and it cannot be glued with itself. 

	A {\em $G^0$-continuous function} on the topological surface $\cM$ 
	is defined by assigning a continuous
	function $f_i$ to each polygon $\sigma_i$, such that the restrictions to the polygonal
	edges are compatible with the homeomorphisms $\mu$.
\end{definition}
The topological surface $\cM$ 
is the disjoint union of the polygons, with some points identified to equivalence
classes by the homeomorphisms $\mu$.
The polygons are also called the faces of $\cM$ and their set is denoted $\cM_{2}$.
Each homeomorphism $\mu$ identifies the edges $\tau_i,{\tau}_j$ of the polygons
$\sigma_{i},\sigma_{j}$ to an {\em interior edge} 
of $\cM$. We say that the edge is shared by the faces
$\sigma_{i}$ and $\sigma_{j}$.
An edge not involved in any homeomorphism $\mu$ is a {\em boundary edge}
of $\cM$. 
The edges of  $\cM$ are the equivalent classes of edges of the
polygons of $\RR^{2}$ identified by the homeomorphisms $\mu$. 
Their set is denoted $\cM_1$.
Similarly, let $\cM_0$ denote the set of $\cM$-vertices, 
that is, equivalence classes of polygonal vertices.
An {\em interior vertex} is an equivalence class of polygonal vertices
$\gamma_0,\gamma_1,\ldots,\gamma_n=\gamma_0$ such that the adjacent
vertices $\gamma_i,\gamma_{i+1}$ are identified by an edge homeomorphism.
The set of  these equivalent classes of identified vertices of the
polygons, or interior vertices of $\cM$, is denoted $\cM_{0}^\circ$.

\subsection{Gluing data}\label{ssec:transition_maps}

The definition of a differential surface $S$
typically requires an atlas of $S$, that is a collection $\{V_p,\psi_p\}_{p\in J}$ such that
$\{V_p\}_{p\in J}$ is an open covering of $S$ \cite{Warner83}.
Each $\psi_p$ is a homeomorphism $\psi_p:U_p\to V_p$, 
where $U_p$ is an open set in $\RR^2$.
For distinct $p,q\in J$ such that $V_p\cap V_q\neq \emptyset$,
let $U_{p,q}:=\psi_p^{-1}(V_p\cap V_q)$ and $U_{q,p}:=\psi_q^{-1}(V_p\cap V_q)$.
Then the map $\psi_q^{-1}\circ\psi_p:U_{p,q}\to U_{q,p}$ is required to be a $C^1$-diffeomorphism. 
The maps $\phi_{pq}:\psi_q^{-1}\circ\psi_p$ are called {transition maps}. 
A differentiable function $f$ on $S$ is a function such that for any open set
$V_{p}$, the composition $f_{p}=f \circ \psi_{p}^{-1}: U_{p} \subset \RR^{2}\to \RR$ is differentiable.

Our objective is to study the space of differentiable functions
that can be constructed on a surface $S$ associated to the topological
surface $\cM$.
Instead of an atlas of a differential surface $S$, we consider a
topological surface $\cM$ together with {\em gluing data} given by maps (that we call {\em transition maps}) between the pairs of faces of $\cM$ that share an edge in $\cM$. 
We make this precise in the following definition.
\begin{definition}\label{def:gluing_data}
	For a topological surface $\cM$, a gluing structure associated to 
	$\cM$ consists of the following:
	\begin{itemize}
	\item for each face $\sigma\in \cM_{2}$ an open set $U_\sigma$ of $\RR^2$
		containing $\sigma$;
	\item for each edge $\tau\in \cM_{1}$ of a cell $\sigma$, an open set 
		$U_{\tau,\sigma}$ of $\RR^2$  containing $\tau$;
 	\item for each edge $\tau\in\cM_{1}$ shared by two faces
 		${\sigma_i},\sigma_j\in \cM_{2}$, a $C^{1}$-diffeomorphism 
 		called the {\em transition map} 
 		$\phi_{\sigma_j,\sigma_i}\colon U_{\tau,\sigma_i}\rightarrow U_{\tau,\sigma_j}$ 
 		between the open sets $U_{\tau,\sigma_i}$ and $U_{\tau,\sigma_j}$, 
 		and also its 	correspondent inverse map $\phi_{\sigma_i,\sigma_j}$;
 	\item for each edge $\tau\in\cM_{1}$ of a cell $\sigma$, the 
 		identity $C^1$-diffeomorphism that defines the identity transition map 
 		$\phi_{\sigma,\tau}=\mathrm{Id}$ between $U_{\sigma}$ and $U_{\tau,\sigma}$.
	\end{itemize}
\end{definition}
 \vspace{-0.5cm}
 \begin{figure}[ht]
	\begin{center}
          \includegraphics[width=4.8cm]{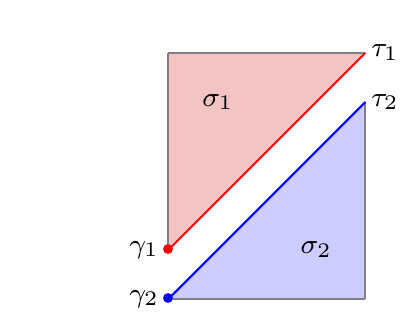}
	\end{center} 
	\vspace{-0.5cm}
	\caption{Topological surface constructed from two triangles.}\label{fig:twotr}
\end{figure}
A transition map as in Definition~\ref{def:gluing_data} differs from
the usual notion of transition map in the context of
differential manifolds (see \cite{Hahn87}), since we do not require compatibility
conditions at the vertices and across edges.
The precise compatibility conditions that we need on these maps
$\phi_{\sigma_j,\sigma_i}$ are given in
Sections~\ref{sec:comp:vertex} ans \ref{sec:topo:restrict}.

\smallskip

Let $\tau=(\tau_{1},\tau_{2})$ be an edge shared by two faces
$\sigma_{1},\sigma_{2} \in \cM_{2}$ and let $\gamma=(\gamma_{1},\gamma_{2})$ be a vertex
of $\tau$ corresponding to $\gamma_{1}$ in $\sigma_{1}$ and to
$\gamma_{2}$ in $\sigma_{2}$, as in Figure~\ref{fig:twotr}. We denote by $\tau_{1}'$
(resp. $\tau'_{2}$) the second edge of $\sigma_{1}$
(resp. $\sigma_{2}$) through $\gamma_{1}$ (resp. $\gamma_{2}$)
We associate to $\sigma_{1}$ and $\sigma_{2}$ two coordinate systems
$(u_{1},v_{1})$ and $(u_{2},v_{2})$ such
that $\gamma_{1}=(0,0)$, $\tau_{1}=\{(u_{1},0), u_{1}\in [0,1]\}$, $\tau'_{1}=\{(0,v_{1}), v_{1}\in [0,1]\}$
and
$\gamma_{2}=(0,0)$, $\tau_{2}=\{(0,v_{2}), v_{2}\in [0,1]\}$, $\tau'_{2}=\{(u_{2},0), u_{2}\in [0,1]\}$.
Using the Taylor expansion at $(0,0)$, a transition map from $U_{\tau,\sigma_1}$ to $U_{\tau,\sigma_2}$ is then of the form 
\begin{equation} \label{eq:transmap}
	\phi_{\sigma_{2},\sigma_{1}}: (u_{1},v_{1}) \longrightarrow  (u_{2},v_{2})=
	\begin{pmatrix}
		v_{1} \,\mb_{\tau,\gamma}(u_{1}) + v_{1}^{2} \rho_{1}(u_{1},v_{1})\\
		u_{1}+v_{1}\,\ma_{\tau,\gamma}(u_{1}) + v_{1}^{2} \rho_{2}(u_{1},v_{1})
	\end{pmatrix}
\end{equation}
where $\ma_{\tau,\gamma}(u_{1}), \mb_{\tau,\gamma}(u_{1}),\rho_{1}(u_{1},v_{1}),\rho_{2}(u_{1},v_{1})$ are $C^{1}$ functions.
We will refer to it as the canonical form of the transition map
$\phi_{\sigma_{2},\sigma_{1}}$ at $\gamma$ along $\tau$. The functions
$[\ma_{\tau,\gamma},\mb_{\tau,\gamma}]$ are called the {\em gluing data} at $\gamma$
along $\tau$ on $\sigma_{1}$. 

\begin{definition}\label{def:crossing}
	An edge $\tau\in\cM$ which contains the vertex $\gamma \in \cM$ 
 	is called a {{\em crossing edge at $\gamma$}} if
	$\ma_{\tau,\gamma}(0)=0$ where $[\ma_{\tau},\mb_{\tau}]$ is the gluing
 	data at $\gamma$ along $\tau$.
 	We define $\cross_{\tau}(\gamma)=1$ if $\tau$ is a 
 	crossing edge at $\gamma$ and $\cross_{\tau}(\gamma)=0$ otherwise.
 	By convention, $\cross_{\tau}(\gamma)=0$ for a boundary edge.
 	If $\gamma\in\cM_0$ is an interior vertex where all adjacent 
 	edges are crossing edges at 
 	$\gamma$, then it is called a {\em crossing vertex}.
 	Similarly, we define $\cross_{+}(\gamma)=1$ if $\gamma$ is a 
 	crossing vertex and $\cross_{+}(\gamma)=0$ otherwise.
\end{definition}

\subsection{Differentiable functions on a topological surface}

We can now define the notion of differentiable function on  $\cM$:
\begin{definition} 
	A differentiable function $f$ on the topological
 	surface $\cM$ is a collection $f=(f_{\sigma})_{\sigma\in \cM}$ of
 	differentiable functions $f_{\sigma}: U_{\sigma}\to \RR$ such that
 	$\forall \gamma \in \tau=\sigma_1\cap\sigma_2$, $\forall \ub \in U_{\tau,\sigma_{1}}$,
	\begin{equation} \label{eq:regcond}
		J_{\gamma} (f_{\sigma_1}) (\ub) = 
		J_{\gamma}(f_{\sigma_2}\circ \phi_{\sigma_2,\sigma_1}) (\ub)
	\end{equation}
	where $J_{\gamma}$ is the jet or Taylor expansion of order $1$ at $\gamma$. 
\end{definition}
If $f_{{1}}, f_{{2}}$ are the functions associated
to the faces $\sigma_{1}, \sigma_{2}\in \cM_{2}$ which are glued along the edge
$\tau$ with a transition map of the form \eqref{eq:transmap}, the
regularity condition \eqref{eq:regcond} leads to the following relations:
\begin{itemize}
	\item $f_1(u_1,0)=f_2\circ \phi_{\sigma_{2},\sigma_{1}}(u_1,0)$ for $u_1\in [0,1]$; that is
		\begin{equation} \label{eq:edgecond0}
 			f_1(u_1,0) = f_2(0,u_1) 
		\end{equation}
	\item $\displaystyle \frac{\partial f_1}{\partial v_1}(u_1,0) 
		=\frac{\partial (f_2\circ \phi)}{\partial v_1}(u_1,0)$ \, for $\phi = \phi_{\sigma_{2},				\sigma_{1}}$ and $u_1\in [0,1]$,
		which translates to 
		\begin{equation} \label{eq:edgecond}
 			\frac{\partial f_1}{\partial v_1}(u_1,0) 
			=\mb_{\tau,\gamma}(u_1)\frac{\partial f_2}{\partial u_2}(0,u_1) 
 			+\ma_{\tau,\gamma}(u_1)\frac{\partial f_2}{\partial v_2}(0,u_1) 
		\end{equation}
		for $u_1\in [0,1]$, with
		$\ma(u_1)=\frac{\partial \phi_1}{\partial v_1}(u_1,0), \ 
		\mb(u_1)=\frac{\partial \phi_2}{\partial v_1}(u_1,0),$ where 
		$\phi_1$ and $\phi_2$ are the components of $\phi$ at the first and the second
		variable respectively.
\end{itemize}
A convenient way to describe this regularity condition is to express the relation
\eqref{eq:edgecond} as a relation between differentials acting on the
space of differential functions on the edge $\tau$:
\begin{equation}\label{eq:differentialrelation}
  \ma_{\tau,\gamma}(u_1)\partial_{ v_2}+
  \mb_{\tau,\gamma}(u_1)\partial_{u_2}
  - \partial_{v_{1}} =0
\end{equation}
With this notation, at a crossing vertex $\gamma$ with $4$ edges we have
$\mb_{\tau,\gamma}(0)\partial_{u_2}  - \partial_{v_{1}} =0$. The differentials along
two opposite edges are ``aligned'', which explains the terminology
of crossing vertex.
\begin{definition}
	A subspace $\cD$ of the vector space of differentiable functions on $\cM$ is
	said to be {\em ample} if at every point $\gamma$ of a face $\sigma$
	of $\cM$, the space of  values and differentials  at $\gamma$, namely 
	$\bigl[ \, f(\gamma), \, \partial_{u_{\sigma}}(f)(\gamma), \, 
	\partial_{v_{\sigma}}f(\gamma) 	\, \bigr]$ for $f\in \cD$, is of dimension $3$.
\end{definition}
This definition does not depend on the choice of the face $\sigma$ to which $\gamma$ belongs,
since for $\gamma$ on a shared edge, the value and differentials
coincide after transformation by the invertible transition map.

\subsection{Compatibility condition at a vertex}
\label{sec:comp:vertex}

Giving gluing data on the edges is not sufficient to ensure
the existence of an ample space of differentiable functions on $\cM$.
At vertices shared by several edges and faces, additional conditions
on the transition maps need to be satisfied.
We describe them in this section, and show that they are
sufficient to construct an ample space of splines on $\cM$ in the following sections.

For a vertex $\gamma \in \cM_{0}^\circ$, (see Fig. \ref{fig:int_vertex}) 
which is common to faces $\sigma_{1}, \ldots,
\sigma_{F}$ glued cyclically around $\gamma$, along the edges $\tau_{i}=\sigma_{i+1}\cap\sigma_{i}$ for $i=1,\ldots, F$ (with 
$\sigma_{F+1}=\sigma_{1}$), we impose the following condition:
\begin{equation}\label{eq:glue-vertex}
	J_{\gamma}(\phi_{1,F})\circ \cdots \circ J_{\gamma}(\phi_{3,2}) 
	\circ J_{\gamma}(\phi_{2,1})(u,v)= (u,v),
\end{equation}
where $J_{\gamma}$ is the jet or Taylor expansion of order $1$ at $\gamma$. 
\begin{figure}[ht]
	\begin{center}
	\includegraphics[width=6cm]{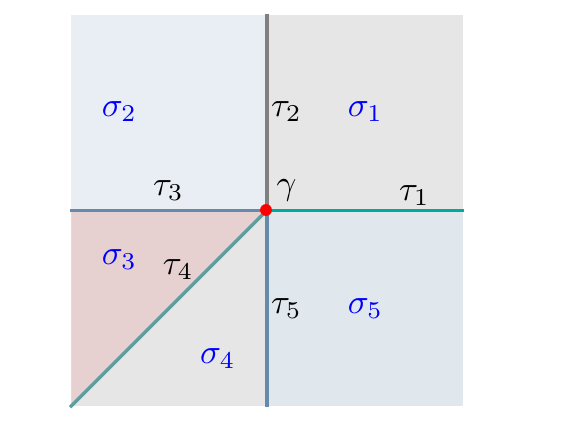}
	\end{center} 
	\vspace{-0.5cm}
	\caption{The faces $\sigma_i$ for $i=1,\dots,5$ are glued cyclically 
	around a vertex $\gamma$.}\label{fig:int_vertex}
\end{figure}
We can assume that for each $i=1, \dots, F$, the edge $\tau_{i}$ is defined 
(linearly) by $v_i=0$ in $\sigma_{i}$. 
It is easy to check that the condition \eqref{eq:glue-vertex} on the Taylor expansion at $\gamma$ leads to the following:
\begin{condition}\label{cond:comp1}
  If the vertex $\gamma$ is on the faces $\sigma_{1},\ldots,\sigma_{F}$
  glued cyclically around $\gamma$, the gluing data $[\ma_{i}, \mb_{i}]$
  at $\gamma$ on  the edges $\tau_{i}$ between
  $\sigma_{i-1}$ and $\sigma_{i}$ satisfies
	\begin{equation}\label{eq:gc1vertex}
		\prod_{i=1}^F \left(
		\begin{array}{cc} 0 & 1 \\ \mb_i(0) & \ma_i(0) \end{array}
		\right) =\left(
		\begin{array}{cc} 1 & 0 \\ 0 & 1 \end{array} \right).
	\end{equation}
\end{condition}
This gives algebraic restrictions on the values $\ma_i(0)$, $\mb_i(0)$.
At a crossing vertex $\gamma$ (see Def.~\ref{def:crossing}) of four incident edges, the equality \eqref{eq:gc1vertex} amounts to
\begin{equation} \label{eq:qbetas}
	\mb_1(0)\mb_3(0)=1,\qquad \mb_2(0)\mb_4(0)=1.
\end{equation}
It turns out that Condition~\ref{cond:comp1} is not sufficient around crossing vertices
for ensuring an ample space of differentiable functions on $\cM$.
An obstruction was noticed in \cite{Peters2010} in a setting 
of rectangular patches. We write this constrain in a general setting: 
\begin{condition}\label{cond:comp2}
  If the vertex $\gamma$ is a crossing vertex with $4$ edges $\tau_{1},\ldots,\tau_{4}$,
  the gluing data $[\ma_{i},\mb_{i}]$ $i=1 \ldots 4$ on these edges at
  $\gamma$ satisfy
	\begin{eqnarray} \label{eq:ddv1}
		\ma'_1(0)+\frac{\mb'_4(0)}{\mb_4(0)} &=& -\mb_1(0) 
		\left(\ma'_3(0)+\frac{\mb'_2(0)}{\mb_2(0)}\right),\\ 
		\label{eq:ddv2}
		\ma'_2(0)+\frac{\mb'_1(0)}{\mb_1(0)} &=& -\mb_2(0) 
		\left(\ma'_4(0)+\frac{\mb'_3(0)}{\mb_3(0)}\right). 
	\end{eqnarray}
\end{condition}
\begin{lemma}
	If the space of differentiable functions on $\cM$ is ample and 
	Condition~\ref{cond:comp1} is satisfied, then the gluing data at every 
	crossing vertex $\gamma$ 	of 4 incident edges must also satisfy 
	Condition~\ref{cond:comp2}.
\end{lemma}
\begin{proof}
The value and first derivatives at every point $\gamma \in \cM$ of all differentiable functions on $\cM$
should span a space of dimension $3$.

If $\gamma$ is a crossing vertex, then we have 4 restrictions on the Taylor expansions
of a spline components $(f_1,f_2,f_3,f_4)$. 
Let us write the Taylor expansion of $f_{i}$ at $\gamma=(0,0)$ as $f_i=p_i+q_iu_i+r_iv_i+s_iu_iv_i+\ldots$.
The gluing conditions imply the following.
From (\ref{eq:edgecond0}), 
\begin{equation}\label{eq:betaq}
	p_1=p_2=p_3=p_4,\quad
	q_1=r_2, \quad q_2=r_3, \quad 
	q_3=r_4 \quad \text{and}\quad q_4=r_1 
\end{equation}
this together with the Condition (\ref{eq:edgecond}) on the first derivatives imply 
\[ 
	\mb_2(0)q_1-q_3=0, \qquad \mb_3(0)q_2-q_4=0.
\]
When we consider the derivative of $f_{i}$ with respect to $u_i$, again
applying (\ref{eq:edgecond}), we get the conditions
\begin{align*}
	s_{2}-\mb_2(0)s_1=\mb_2'(0)q_1+\ma_2'(0)q_2, \\
	s_{3}-\mb_3(0)s_2=\mb_3'(0)q_2+\ma_3'(0)q_3, \\
	s_{4}-\mb_4(0)s_3=\mb_4'(0)q_3+\ma_4'(0)q_4,\\
	s_{1}-\mb_1(0)s_4=\mb_1'(0)q_4+\ma_1'(0)q_1.
\end{align*}
By combining the last four equations respectively with the weights 
$\mb_4(0)$, $\mb_1(0)\mb_4(0)$, $\mb_1(0)$, $1$, together with Condition~\ref{cond:comp1}, 
we get
\begin{align} \label{eq:unwanted}
	q_1 \,\bigl(\ma_1'(0)+\mb_2(0)\mb_4'(0)\bigr) +
	q_3 \, \mb_1(0) \bigl(\mb_2(0)\ma_3'(0)+\mb'_2(0)\bigr) + \qquad\qquad \nonumber \\
	q_2 \, \bigl(\mb_1(0)\ma_2'(0)+\mb_1'(0)\bigr) +
	q_4 \, \mb_2(0)\bigl(\ma_4'(0)+\mb_1(0)\mb'_3(0)\bigr) = 0.
\end{align}
This relation does not involve the cross derivatives $s_1,s_2,s_3,s_4$,
but gives an unwanted relation between the first order derivatives.
After replacing (\ref{eq:betaq}) and (\ref{eq:qbetas}) in \eqref{eq:unwanted}, 
we encounter conditions (\ref{eq:ddv1}) and (\ref{eq:ddv2}). 
Under these conditions, there is no relation between $q_1,q_2$,
and there is one degree of freedom for $(s_1,s_2,s_3,s_4)$.
\end{proof}

The restrictions (\ref{eq:ddv1}) and (\ref{eq:ddv2}) were noticed in \cite{Peters2010}
in the context of gluing tensor product rectangular patches with all $\mb_{i}(0)=-1$. 
The restrictions are then simply 
\begin{equation*} 
	\ma'_1(0)=\ma'_3(0), \qquad \ma'_2(0)=\ma'_4(0).
\end{equation*}

\subsection{Topological restrictions}
\label{sec:topo:restrict}

A guiding principle for the construction of geometric continuous functions is that
$G^{1}$ properties are equivalent to $C^{1}$ properties in the plane
after an adequate reparameterization of the problem. 
Gluing two faces along an edge is transformed locally via such 
reparameterization maps, into gluing two
half-planes along a line.
Each half-plane is in correspondence with the half-plane
determined by one of the faces and the shared edge.
A natural gluing is to have the half planes on 
each side of the line. In this case,  the points of one face are mapped
by the reparameterizations on one side of the line and the points of
the other face on the other side of the line.
This implies that the transition maps keep locally the points of a
face on the same side of the edge and thus it should have a positive Jacobian at each
point of the edge.
Therefore the first topological restriction that we ask for each edge $\tau$, using the
canonical form \eqref{eq:transmap},  is the following:
\[
\forall u\in [0,1],\  \mb_{\tau}(u)<0.
\]
When the function $\mb_{\tau}$ is positive on the edge; 
the transition map identifies interiors of the polygons.
It corresponds to two patches of surfaces virtually pasted at a sharp edge (i.e., at angle 0)
rather than in a proper continuously smooth manner (i.e., at the angle $\pi$).
In some CAGD applications, it may be useful to model surfaces with sharp wing-like
edges by the $G^1$ continuity restrictions with $\mb_{\tau}>0$.  
But typical $G^1$ continuity applications should require $\mb_{\tau}<0$ on the whole edge
to prevent this degeneration.  
The regularity property across edges, considered 
irrespective of orientation, is called weak geometric continuity \cite[\S 6.7]{derose}
and the restriction we consider corresponds to coherently oriented parametrizations
in \cite[\S 6.8]{derose}. 

Similarly, gluing the faces around a vertex $\gamma$ should be equivalent to
gluing sectors around a point in the plane, via the reparameterization
maps.
Such sectors should form a fan around the parameter
point, which can be identified with the local neighborhood of the vertex $\gamma$
on the surface. Thus these sectors should not overlap. 
If this fan is defined by vectors $\ub_{1},\ldots, \ub_{F'}\in
\RR^{2}$ ($\ub_{i+1}$ is supposed to be outside the union of the sectors defined by two consecutive
vectors $\ub_{j-1},\ub_{j}$ for $2<j<i$), we easily check that the
coefficients $\ma_{i}(0)$, $\mb_{i}(0)$ of the transition map
\eqref{eq:transmap} across the edge $\tau_{i}$ at $\gamma$ are such that:
\begin{equation}\label{eq:relui}
	\ub_{i-1} = \ma_{i}(0) \ub_{i} + \mb_{i}(0)\, \ub_{i+1}
\end{equation}
or equivalently
\[ 
	[\ub_{i},\ub_{i-1}]
 	=\left[ \begin{array}{cc}0& 1\\
 	\mb_{i}(0)& \ma_{i}(0)\end{array}\right]
	[\ub_{i+1},\ub_{i}] 
\]
(see also the construction in the next section \ref{sec:gluingexpl}).
If the sector angles are less than $\pi$ (i.e. the sector
$\ub_{i},\ub_{i+1}$ coincides with the cone generated by $\ub_{i},\ub_{i+1}$)
the condition that the sectors form a fan and do not overlap translates as follows:
the coefficients of the last row of 
\begin{equation*}
	\prod_{i=j}^{k} \left(
	\begin{array}{cc} 0 & 1 \\ \mb_i(0) & \ma_i(0) \end{array}
	\right) 
\end{equation*}
should not be both non-negative for $1<j\le k<F$.

A natural way to define transition maps at a vertex $\gamma$ which satisfy
this condition is to choose vectors in the plane that define a fan, as
in Figure \ref{fig:int_vertex}. Then the coefficients $\ma_{i}(0)$,
$\mb_{i}(0)$ are uniquely determined from the relations \eqref{eq:relui}.

The topological constraints 
could be dropped in some applications, for example, 
when modeling analytical surfaces with branching points, 
or surfaces with sharp wing-like interior edges,
or with winding-up boundary. In these specific applications,
the compatibility Condition \ref{cond:comp2} at {\em crossing vertex} might need to be extended,
to allow winding up of 8, 12, etc., crossing edges, and take into account 
the sharp edges.
Apart from this kind of consideration, the topological conditions do not essentially affect our algebraic
dimension count.

The framework that we propose is more general than previous 
approaches used in Geometric Modeling to define $G^{1}$ splines (see
e.g. \cite[\S 3]{Peters:2008})
since it allows to define differentiable functions on topological
surfaces such as a M\"obius strip or a Klein bottle.

Moreover, it does not rely on the construction of manifold surfaces and 
atlas, but only on compatible transition maps.

\subsection{Example}\label{sec:gluingexpl}
A simple way to define transition maps is to use a symmetric gluing
as proposed in \cite [\S 8.2]{Hahn87} for rectangular patches. If $\tau= (\gamma_{0},\gamma_{1})$ is
the shared edge between $\sigma_{1}$ and $\sigma_{2}$, the
transition map can be of the form:
\begin{equation}\label{eq:hahngluing}
	\phi (u,v)= \left (
	\begin{array}{c}
		-v \\
		 u + 2v \bigl(\md_{0} (u) \cos {\frac{2\pi}{n_{0}}} 
		- \md_{1} (u) \cos \frac{2\pi}{n_{1}} \bigr) 
	\end{array}
	\right)
\end{equation} 
where $n_{0}$ (resp. $n_{1}$) is the number of edges at the vertex $\gamma_{0}$ (resp.
$\gamma_{1}$). Additionally, if $\gamma_{0}$ corresponds to $u=0$ and
$\gamma_{1}$ to $u=1$, the functions $\ma$ and $\mb$
interpolate $0$ and $1$: $\md_{0}(0)=1$, $\md_{0} (1)=0$, $\md_{1} (0) =0$, $\md_{1}(1)=1$
and their derivatives of order 1 should vanish at $0$, $1$. It
corresponds to a symmetric gluing, where the angle of two consecutive
edges at $\gamma_{i}$ is $\frac{2\pi}{n_{i}}$.
If $\md_{0} (u)$ and $\md_{1} (u)$ are polynomial functions, their degree must be at
least $3$.
If $\md_{0} (u)$ and $\md_{1} (u)$ are rational functions with the
same denominator, the maximal degree of the numerators and denominator must be
at least $2$.
As we will see the dimension of the spline space decreases when the
degree increases. Thus it is important to construct transition maps with low degree
numerators and denominators. See e.g. \cite{PetersHandbook,Peters2010} for low degree
constructions, which depend on the structure of $\cM$.

A general construction of gluing data which satisfies the
compatibility conditions is as follows.
\begin{enumerate}
	\item For all the vertices $\gamma\in \cM_{0}$ and for all the edges
  	$\tau_{1},\ldots, \tau_{F}$ of $\cM_{1}$ that contain $\gamma$, choose vectors
   	$\ub_{1},\ldots,\ub_{F}\in \RR^{2}$ such that the cones generated by
   	$\ub_{i},\ub_{i+1}$ form a fan in $\RR^{2}$  and such that the
   	union of these cones is $\RR^{2}$ when $\gamma$ is an interior vertex.
   
   	Compute the transition map $\phi_{\sigma_{i},\sigma_{i-1}}$ at 
   	$\gamma=(0,0)$ on the edge $\tau_{i}$
		\[
			\phi_{\sigma_{i},\sigma_{i-1}}(0,0) = S \circ [\ub_{i},\ub_{i+1}]^{-1}\circ
			[\ub_{i-1},\ub_{i}] \circ S
			=\left[ \begin{array}{cc}0&\mb_{\tau_{i}}(0)\\ 
			1& \ma_{\tau_{i}}(0)\end{array}\right]
		\]
		where $S=\left[ \begin{array}{cc}0&1\\1&0\end{array}\right]$,
		$[\ub_{i},\ub_{i+1}]$ is the matrix which columns are the vectors
		$\ub_{i}, \ub_{i+1}$, $|\ub_{i},\ub_{j}|$ is the determinant of the
		vectors $\ub_{i},\ub_{j}$	and 
		\[
			\ma_{\tau_{i}}(0) = \frac{|\ub_{i-1},\ub_{i+1}|}{|\ub_{i},\ub_{i+1}|},
			\mb_{\tau_{i}}(0) =-\frac{|\ub_{i-1},\ub_{i}|}{|\ub_{i},\ub_{i+1}|}.
		\]
	\item For all the edges $\tau\in \cM_{1}$, define the rational functions	
		$\ma_{\tau}=\frac{a_{\tau}}{c_{\tau}},\mb_{\tau}=\frac{b_{\tau}}{c_{\tau}}$ on the
		edges $\tau$ by interpolation as follows: 
		if there is no crossing edge in $\cM_{1}$, then a linear interpolation of the value
		at the vertices is sufficient.
		If $\gamma_1,\ldots,\gamma_{n}$ is a sequence of crossing vertices,
		and $\gamma_0\gamma_1,\gamma_1\gamma_2,\ldots,\gamma_{n}\gamma_{n+1}$ 
		is a sequence of edges passing ``straight" through them, we can choose linear gluing data 
		on one edge, and quadratic data on the remaining edges of the sequence
		so that the constraints (\ref{eq:ddv1}) and (\ref{eq:ddv2}) are
		satisfied.
\end{enumerate}
Therefore, for general meshes, gluing data which satisfy the
compatibility condition and the topological condition can be
constructed in degree $\le 2$.

\section{Spline space on a topological surface}

The main object of our study is the space of functions on the topological surface
$\cM$, which are differentiable and piecewise polynomial. Such
functions are called {\em spline functions} on $\cM$.
Let $\cR(\sigma)=\RR[u_{\sigma},v_{\sigma}]$ be the ring of
polynomials in the variables $(u_{\sigma},v_{\sigma})$ attached to the
face $\sigma$.
A spline function $f$ is defined by assigning to each face $\sigma\in
\cM_{2}$ a polynomial $f_{\sigma}\in \cR(\sigma)$, and by imposing the
regularity conditions across the shared edges.

We also consider rational gluing data on the interior edges $\tau\in \cM_{1}$:
\begin{equation}\label{eq:edgeabc} 
	\ma_{\tau}(u_{1})=\frac{a_{\tau} (u_{1})}{c_{\tau} (u_{1})} \quad 
 	\text{and\quad}\mb_{\tau}(u_{1}) 
	=\frac{b_{\tau} (u_{1})}{c_{\tau} (u_{1})}
\end{equation}
with $a_{\tau} (u_{1}), b_{\tau} (u_{1})$ and $c_{\tau} (u_{1})$ polynomials in the variable 
$u_{1}$, where $b_{\tau}(u_{1})$ and $c_{\tau}(u_{1})$ do not vanish
on $\tau$ (i.e. for $u_{1}\in [0,1]$).
As $b(u_{1}), c(u_1)$ do not vanish on $\tau$, the transition map $\phi_{{\sigma_{2}},\sigma_1}$ 
is a $C^{1}$-diffeomorphism in a neighborhood of the edge $\tau=(\tau_{1},\tau_{2})$ between $\sigma_1$ and $\sigma_1$. 
The polynomial vector $[a_{\tau},b_{\tau},c_{\tau}]$ is also
called hereafter the gluing data of  the edge $\tau$. 
We assume hereafter that {\em the transition maps satisfy Conditions
\ref{cond:comp1},  \ref{cond:comp2} and all crossing vertices of $\cM$
have 4 edges}.

We can now define the space of splines on $\cM$:
\begin{definition}\label{def:splinesp}
	We denote by $\cS^{1}(\cM)$ the ($\RR$-linear) space of
	differentiable functions on the topological surface 
	$\cM$ which are defined by assigning polynomials 
	to the faces  $\sigma\in \cM_{2}$ satisfying the $G^{1}$ constrains
	\eqref{eq:regcond}. More precisely,
	\begin{equation*}
		\cS^{1}(\cM):=\{f\in \oplus_{{\sigma}\in \cM_{2}}\cR ({\sigma}) \mid
		J_{\gamma} (f_{\sigma_1})  = J_{\gamma}(f_{\sigma_2}\circ \phi_{\sigma_2,\sigma_1}) 
		\ \forall \gamma \in \tau=\sigma_1\cap\sigma_2\},
	\end{equation*}
	where $J_{\gamma}$ is the jet or Taylor expansion of order $1$.
\end{definition}
A  spline $f\in\cS^{1}(\cM)$ gives a piecewise-polynomial map, 
defined on every face of $\cM$, and the jets of order ${1}$ coincide on the shared edges. 
This definition implies that a spline function $f\in \cS^{1}(\cM)$
is $C^{{1}}$ on a neighborhood of a shared edge $\tau=\sigma_1\cap \sigma_2$ if we use the
re-parametrization $f_{\sigma_2} \circ \phi_{\sigma_2,\sigma_1}$.

\smallskip

Definition \ref{def:splinesp} can be directly extended to splines of any order $r$
but in this paper we only consider $r=1$.

\subsection{Polynomials on faces}
On each face ${\sigma}\in \cM_{2}$, we consider polynomials of degree bounded
by $k\in \NN$.

If $\sigma\in \cM_{2}$ is a triangle $(P,Q,R)$, we denote by $\cR_{k} ({\sigma})$ 
the finite dimensional vector space of polynomials in 
$\RR[u_{{\sigma}},v_{{\sigma}}]$ with total degree bounded by $k$.

After a change of coordinates, we may assume that the coordinate function $u=u_{\sigma}$ satisfies $u_\sigma(PR)=0$ 
and $u_{\sigma}(Q)=1$, while $v=v_{\sigma}$ satisfies $v_{\sigma}(PQ)=0$ and $v_{\sigma}(R)=1$.
Introducing $w=1-u-v$, we can express any polynomial in
$\cR_{k} ({\sigma})$ as a homogeneous polynomial 
of degree $\le k$ in the {\em barycentric} coordinates $u,v,w$
using the {\em Bernstein-B\'ezier basis}:
\[
	b_{i,j}^{^{\Delta}}(u,v,w)=\frac{k!}{i!j!(k-i-j)!} \, u^i \, v^j \, w^{k-i-j},
\]
for $0\le i+j \le k$.
We verify directly that for a function
\[ 
	f(u,v) = \sum_{0\le i+j\le k} c_{i,j} b_{i,j}^{^{\Delta}}(u,v,1-u-v)
\]
expressed in this basis, we have
\begin{eqnarray}\label{eq:TaylorBernsteinTriang}
	f(0,0) &=& c_{0,0},\\ \nonumber
 	\partial_{u}f(0,0)&= &k (c_{1,0}-c_{0,0}),\ \ \ 
 	\partial_{v}f(0,0)= k(c_{0,1}-c_{0,0}),\\ \nonumber
 	\partial_{u} \partial_{v}f(0,0)&=& k\, (k-1)\, (c_{1,1}-c_{1,0}-c_{0,1} + c_{0,0}). 
\end{eqnarray}
If $\sigma\in \cM_{2}$ is a rectangle $(P,Q,R,S)$,
we will denote by $\cR_{k}({\sigma})$ the finite dimensional vector space of
polynomials in $\RR[u_{{\sigma}},v_{{\sigma}}]$ with partial degree in
$u_{\sigma}$ and $v_{\sigma}$ bounded by $k$, where 
$u=u_\sigma$ is chosen such that $u_{\sigma}(PS)=0$, $u_{\sigma}(QR)=1$, 
and $v=v_{\sigma}$ is chosen such that $v_{\sigma}(PQ)=0$, $v_{\sigma}(RS)=1$.
Introducing $\tilde{u}=1-u$, $\tilde{v}=1-v$, we can express
any polynomial function of $\cR_{k}(\sigma)$ as a bi-homogeneous polynomial of degree $k$ 
in $u,\tilde{u}$ and degree $k$ in $v,\tilde{v}$, using
the {\em tensor product Bernstein-B\'ezier basis}
\[
	b_{i,j}^{^{\Box}}(u,\tilde{u},v,\tilde{v})=\frac{k!k!}{i!j!(k-i)!(k-j)!} \, u^i 
	\,\tilde{u}^{k-	i} \, v^j \, \tilde{v}^{k-j}.
\]
for $0\le i\le k$, $0\le j\le k$.
We verify directly that for a function $f= \sum_{0\le i,j \le k} c_{i,j}
b_{i,j}^{^{\Box}}$ expressed in this basis, we have
\begin{eqnarray}\label{eq:TaylorBernsteinQuad}
	f(0,0)&=& c_{0,0},\\ \nonumber
 	\partial_{u}f(0,0) &=& k(c_{1,0}-c_{0,0}), \ \ \ 
 	\partial_{v}f(0,0) = k(c_{0,1}-c_{0,0}),\\ \nonumber
 	\partial_{u} \partial_{v}f(0,0)&=& k^{2} \,(c_{1,1}-c_{1,0}-c_{0,1} + c_{0,0}).
\end{eqnarray}
The finite dimensional vector space of spline functions
$f=(f_{\sigma})_{\sigma\in \cM_{2}} \in \cS(\cM)$ of degree bounded by
$k\in \NN$ on each face ($f_{\sigma}\in
\cR_{k}(\sigma)$) and of regularity $r$ is denoted $\cS_{k}^{r}(\cM)$
or simply $\cS_{k}(\cM)$ when $r=1$.

\subsection{Taylor maps}
An important tool that we are going to use intensively is the Taylor
map associated to a vertex or to an edge of $\cM$.

Let $\gamma\in \cM_{0}$ be a vertex on a face $\sigma\in
\cM_{2}$ belonging to two edges $\tau, \tau' \in \cM_{1}$ of
$\sigma$. We define the {\em ring of $\gamma$ on $\sigma$} by 
 $\cR^{\sigma}(\gamma)= \cR(\sigma)/(\ell_{\tau}^{2}, \ell_{\tau'}^{2})$
where $(\ell_{\tau}^{2}, \ell_{\tau'}^{2})$ is the ideal generated by 
the squares of $\ell_\tau$ and $\ell_{\tau'}$, the equations $\ell_{\tau}(u,v)=0$ and $\ell_{\tau'}(u,v)=0$
are respectively the equations of $\tau$ and $\tau'$ in $\cR(\sigma)=\RR[u,v]$.

The {\em Taylor expansion at $\gamma$ on $\sigma$} is the map
\[
	T_{\gamma}^{\sigma}: f \in \cR(\sigma) \mapsto f \mod 
	(\ell_{\tau}^{2}, \ell_{\tau'}^{2}) \text{\; in\; } \cR^{\sigma}(\gamma).
\]
Choosing an adapted basis of $\cR^{\sigma}(\gamma)$, one can define $T_{\gamma}^{\sigma}$ by 
\[
	T_{\gamma}^{\sigma}(f) = \bigl[\, f(\gamma),\, \partial_{u} f(\gamma), \, 
	\partial_{v} 		f(\gamma),\, \partial_{u}\partial_{v} f(\gamma)\, \bigr]. 
\]
The map $T_{\gamma}^{\sigma}$ can also be defined in  another basis
of $\cR^{\sigma}(\gamma)$ in terms of the Bernstein coefficients by
\[ 
	T_{\gamma}^{\sigma}(f) = \bigl[\, c_{0,0}(f),\,  c_{1,0}(f), \, c_{0,1}(f),\,  
	c_{1,1}(f) \, \bigr]
\]
where $c_{0,0}, c_{1,0}, c_{0,1}, c_{1,1}$ are the first Bernstein coefficients associated to $\gamma=(0,0)$.

We define the Taylor map  $T_\gamma$ on all the faces $\sigma$ that contain $\gamma$, 
\[
	T_{\gamma}: f=(f_{\sigma})\in \oplus_{\sigma}\cR(\sigma)
	\rightarrow (T_{\gamma}^{\sigma}(f_{\sigma}))\in
	\oplus_{\sigma\supset\gamma} \cR^{\sigma}(\gamma).
\]
Similarly, we define $T_{0}$ as the Taylor map at all the vertices on
all the faces of  $\cM$. 

\smallskip

For an edge $\tau\in \cM_{1}$ on a face $\sigma\in \cM_{2}$, 
we define the {\em ring of $\tau$ on $\sigma$} by $\cR^{\sigma}(\tau)= \cR(\sigma)/(\ell_{\tau}^{2})$
where $\ell_{\tau}(u,v)=0$
is the equation of $\tau$ in $\cR(\sigma)=\RR[u,v]$.
The {\em Taylor expansion along $\tau$ on $\sigma$} is defined by 
\[
	T_{\tau}^{\sigma}: f \in \cR(\sigma) \mapsto f \mod 
	(\ell_{\tau}^{2}) \text{\, in }\cR^{\sigma}(\tau),
\]
and the Taylor map on all the faces $\sigma$ that contain $\tau$ 
is given by 
\[
	T_{\tau}: f=(f_{\sigma})\in \oplus_{\sigma}\cR(\sigma)
	\rightarrow (T_{\tau}^{\sigma}(f_{\sigma}))\in
	\oplus_{\sigma\supset\tau} \cR^{\sigma}(\gamma).
\]
Similarly, we define $T_{1}$ as the Taylor map along all the edges on
all the faces of $\cM$.

\section{$G^{1}$ splines along an edge}

To analyze the constraints imposed by gluing data along an edge, we
consider first a simple topological surface $\cN$ composed of two 
faces $\sigma_{1}, \sigma_{2}$ glued along an edge $\tau$. 
  
A spline function $f\in S_{k}^{1}(\cN)$ on $\cN$ is represented by a pair of
polynomials $f=(f_{1},f_{2})$ with $f_{i}\in \cR(\sigma_{i})=\RR[u_{i},v_{i}]$ for $i=1,2$.

By a change of coordinates, we assume that the edge $\tau$ is defined
by $v_{1}=0$ and $u_{1}\in [0,1]$ in $\sigma_{1}$ and by $u_{2}=0$ and
$v_{2}\in [0,1]$ in $\sigma_{2}$.

\subsection{Splines and syzygies}\label{eq:edgespace}

With the transition map $\phi_{\sigma_2,\sigma_1}$ defined by  the rational
functions $\ma=\frac{a_{\tau}}{c_{\tau}}$ and
$\mb=\frac{b_{\tau}}{c_{\tau}}$ as in \eqref{eq:edgeabc}, the
differentiability Condition (\ref{eq:edgecond}) along the interior edge
$\tau$ becomes
\begin{equation*}
	a(u_1)A(u_1)+b(u_1)B(u_1)+c(u_1)C(u_1)=0,
\end{equation*}
where
\[
	A(u_1)=\frac{\partial f_2}{\partial v_2}(0,u_1), \quad
	B(u_1)=\frac{\partial f_2}{\partial u_2}(0,u_1), \quad
	C(u_1)=-\frac{\partial f_1}{\partial v_1}(u_1,0).
\]
Thus, the $G^1$-smoothness condition along an interior edge 
is equivalent to the condition on $(A,B,C)$ of being a 
{\em syzygy} of the polynomials $a(u_1), b(u_1), c(u_1)$.

\bigskip

More precisely, a $G^1$ spline $(f_1, f_ 2)$ on $\cN$ is constructed from a syzygy $(A,B,C)$
of $a, b, c$ by defining:
\begin{align} \label{eq:edgeint1}
	f_1(u_1,v_1)=c_0+\int_0^{u_1} \! A(t)dt-v_1C(u_1)+v_1^2E_1(u_1,v_1),\\
 	\label{eq:edgeint2}
	f_2(u_2,v_2)=c_0+\int_0^{v_2} \! A(t)dt+u_2B(v_2)+u_2^2E_2(u_2,v_2),
\end{align}
where $c_0\in\RR$ is any constant, and $E_1,E_2$ are (any) polynomials in $\RR[u_i,v_i]$ for $i=1,2$, respectively.

We will use this representation for the splines on $\cN$ to compute
the dimension of the space of $G^1$ splines $\cS_k^1(\cN)$, see Proposition \ref{prop:dimN} below. 
Before, we introduce some notation, both for the proof and the dimension formula. 

\medskip

The module of syzygies of $a(u_1), b(u_1), c(u_1)$ over the ring
$\RR[u_1]$ is denoted by $Z=\Syz(a,b,c)$. For $(A,B,C)\in Z$, the maximum of the degrees, 
$\max(\deg A$, $\deg B$, $\deg C)$ is called 
the {\em coefficient degree} of the syzygy. 

Each of the faces $\sigma_1$ and $\sigma_2$ in $\cN$ can be a triangle or 
a rectangle. Let us denote by $F_{\Box}$ the number of rectangles
and by $F_{\Delta}$ the number of triangles in $\cN$.
\begin{definition} \label{def:Fm}
	As before, let $\sigma_{1},\sigma_{2}$ be the faces of $\cN$.
	We define 
	\[
		m=\min\bigl(F_{\Delta}(\sigma_{1}) ,F_{\Delta}(\sigma_{2})\bigr),
	\]
	where $F_{\Delta}(\sigma_{i})=1$ if $\sigma_i$ is a triangle and
	$0$ otherwise. For the polynomials $a,b,c\in \RR[u_1]$ defining 
	the gluing data along the edge $\tau$, let 
	$ n = \max\bigl(\deg(a), \deg(b), \deg(c)\bigr)$, 
	\[
		d_{a} = n + 1,\quad 
		d_{b} = n + F_{\Delta}(\sigma_{2}),
		\quad \text{and}\quad 
		d_{c} = n + F_{\Delta}(\sigma_{1}),
	\]
	and
	\[
	e = 
	\begin{cases}
		0\,, &\mbox{\; if 
		$ \min\bigl( d_{a} - \deg(a), d_{b} - \deg(b), d_{c} - 
		\deg(c) \bigr) = 0 $\, and} \\ 
		1\,, &\mbox{\; otherwise.}\\ 
	\end{cases}
	\]
\end{definition}

By the formulas \eqref{eq:edgeint1} and \eqref{eq:edgeint2}
representing a spline $(f_1,f_2)\in\cS^1_k(\cN)$, let us notice that
we need to consider syzygies $(A,B,C)$ of $a,b,c$ such that
$\deg(A)\le k-1$, $\deg(B)\le k-F_{\Delta}(\sigma_{2})$,
and $\deg(C)\le k-F_{\Delta}(\sigma_{1})$. The reason is that $f_{i}$ is of
bidegree at most $(k,k)$ if $\sigma_{i}$ is a rectangle and of total
degree at most  $k$ if $\sigma_{i}$ is triangle.
\begin{definition}\label{def:Zk}
	For $k\geq 0$, we will denote by $Z_k$ the vector subspace of 
	$Z$ of syzygies of $(a, b, c)$ defined as the set
	\begin{align*}
		Z_k = \{ (A, B, C)\in Z \colon \deg(A)\leq & k-1, \; \deg(B)\le k-F_{\Delta}(\sigma_{2}),\\
		&\text{\, and\, } \deg(C)\le k-F_{\Delta}(\sigma_{1})\}.
	\end{align*}
\end{definition}
Let us consider the map
\begin{eqnarray}\label{eq:deftheta}
	\Theta_\tau : Z &\rightarrow & \cS^{1}(\cN) \\
 	(A,B,C) & \mapsto& \biggl(\int_0^{u_1} \! A(t)dt-v_1C(u_1), \, 
 	\int_0^{v_2} \! A(t)dt+u_2B(v_2) \biggr).\nonumber
\end{eqnarray}
By construction, we have $\Theta_\tau(Z_{k})\subset \cS_{k}^{1}(\cN)$.

The dimension of $Z_k$, as a vector space over $\RR$, will be
deduced from classical results on graded modules over $S=\RR[u_{0},u_1]$. 
We will study the module $\Syz(\bar{a}, \bar{b}, \bar{c})$, where
$\bar{a}, \bar{b}, \bar{c}\in S$ are the homogenization of $a, b$ and $c$ in degree $d_a, d_b$ and $d_c$ respectively. 
The elements in $\Syz(\bar{a}, \bar{b}, \bar{c})$ in degree $n+k$ will precisely lead to the syzygies in $Z_k$. 
\begin{lemma} \label{lm:syzygy}
	For polynomials $a, b, c \in\RR[u_1]$, with $b,c\neq 0$, $\gcd(a,b,c)=1$ 
	and $Z=\Syz(a, b, c)$ as defined above,
	\begin{enumerate}
	\item $Z$ is a free $\RR[u_1]$-module of rank $2$. 
	\item The module $Z$ is generated by vectors
		$(A_1,B_1,C_1)$, $(A_2,B_2,C_2)$ of coefficient degree 
		$\mu$ and $\nu = n - \mu+1 +F_{\Delta}-e -2 m$
		where $\mu$ is the smallest possible coefficient degree.
	\item For $k\in \NN$, the dimension of $Z_{k}$ as vector space over $\RR$ is given by
		\[
 			\dim Z_{k} = (k-\mu-m+1)_{+} + (k-n+\mu+m-F_{\Delta}+e)_{+}
		\]
		where $t_{+}=\max(0,t)$ for $t\in \ZZ$.
	\item The generators $(A_1,B_1,C_1)$, $(A_2,B_2,C_2)$ of $Z$ can be
	 chosen so that
		\[
			(a,b,c)=(B_1C_2-B_2C_1, C_1A_2-C_2A_1, A_1B_2-A_2B_1). 
		\]
	\end{enumerate}
\end{lemma}
\begin{proof}
We study the syzygy module $Z = \Syz(a,b,c)$ using results on graded resolutions. 
For this purpose, we homogenize 
$a, b$ and $c$ in degree $d_{a}=n+1$,
$d_{b}=n+F_{\Delta}(\sigma_{2})$, 
and $d_{c}=n+F_{\Delta}(\sigma_{1})$, respectively, 
where $F_\Delta(\sigma_i)$ is as in Definition \ref{def:Fm}. 
Let $u_0, u_1$ be the homogeneous coordinates, and 
$\bar{a}, \bar{b}, \bar{c}$ the corresponding homogenizations
of $a$, $b$, and $c$. We consider the module of homogeneous syzygies 
$\Syz(\bar{a}, \bar{b}, \bar{c})$
over the polynomial ring $S = \RR[u_0,u_1]$.
\begin{claim}\label{claim:Zk}
	For any $k\geq 0$, the elements in $Z_k$ are 
	exactly the syzygies of degree $n+k$ in $\Syz(\bar{a}, \bar{b}, \bar{c})$
	after dehomogenization by setting $u_0 = 1$.
\end{claim}
\begin{proof}
It is clear that if $\bar{A}\bar{a} + 
\bar{B}\bar{b} + \bar{C}\bar{c} = 0$, then by dehomogenization taking $u_0 = 1$, 
we get a syzygy $(A, B, C)$ of $(a, b, c)$. 
Moreover, if 
$\deg(\bar{A}\bar{a}) = \deg(\bar{B}\bar{b}) = \deg(\bar{C}\bar{c}) =
n + k $, then 
$\deg(\bar{A}) = k - 1 $, $\deg(\bar{B}) = k - F_\Delta (\sigma_2)$ and
$\deg(\bar{C}) = k - F_\Delta (\sigma_1)$. 
It follows that $(A, B, C)\in Z_k$. 

On the other hand, any syzygy $(A,B,C)\in Z_k$ is given by polynomials
that satisfy the conditions in Definition \ref{def:Zk}. 
Thus $\max \{\deg A, \deg B, \deg C\} \leq k$, and since $ n =\max\{\deg a, \deg b,
\deg c\}$ then we may consider the homogenization 
of the polynomial $Aa + Bb +Cc$ in degree $n+k$. 
These polynomials satisfy
\begin{align*}
	0 & = u_0^{k+n} (A a + B b + C c) (u_1/u_0\bigr) \\
 	& = u_0^{k-1} \cdot u_0^{n+1} Aa\bigl(u_1/u_0\bigr) + u_0^{k - F_{\Delta}   
 	(\sigma_2)} \cdot 	u_0^{n+F_{\Delta}(\sigma_2)} Bb\bigl(u_1/u_0\bigr) \\
 	&\hspace{3.88cm} + u_0^{k - F_{\Delta}(\sigma_1)} \cdot u_0^{n+F_{\Delta} (\sigma_1)} 
 	Cc\bigl(u_1/u_0\bigr).
\end{align*}
It is easy to check that 
\[
	\bar{A} = u_0^{k-1}A(u_1/u_0),
	\quad \bar{B} = u_0^{k-F_\Delta(\sigma_2)}B(u_1/u_0), 
	\quad \bar{C} = u_0^{k - F_\Delta(\sigma_1)}C(u_1/u_0)
\] 
are all polynomials in $\RR[u_1,u_0]$, and define a syzygy of 
$\bar{a}, \bar{b}, \bar{c}$ of degree $n+k$. 
Let us also notice that the polynomials 
\[
	\bar{a} = u_0^{n+1}a(u_1/u_0), 
	\; \bar{b} = u_0^{n + F_\Delta(\sigma_2)}b(u_1/u_0), \; \, 
 	\text{and} 
 	\; \bar{c} = u_0^{n + F_\Delta(\sigma_1)}c(u_1/u_0)
\]
are precisely the homogenization
of $a, b, c$ in degree $d_a, d_b, d_c$, respectively. 
\end{proof}

As $gcd(a,b,c)=1$, we have 
$gcd(\bar{a}, \bar{b}, \bar{c})=u_{0}$ if $e=1$,
and $gcd(\bar{a}, \bar{b}, \bar{c})=1$ otherwise.

Let $I=(\bar{a},\bar{b},\bar{c})$ be the ideal generated by 
$\bar{a},\bar{b},\bar{c}$ in $S$. 
If $gcd(\bar{a}, \bar{b}, \bar{c}) = 1$ then there exists $t_{0}\in \NN$ 
such that $\forall t\ge t_{0}$, $I_{t}= (u_{0},u_{1})^{t}$
and in that case, $\dim_\RR (S/I)_t = 0$ for $t$ sufficiently large. 
It follows that the Hilbert polynomial $HP_{S/I}$ of $S/I$ 
is the zero polynomial. 

For the second case, namely if $gcd(\bar{a}, \bar{b}, \bar{c}) = u_0$, 
since $\gcd(a,b,c) = 1$ then the polynomials $\bar{a}/u_0$, $\bar{b}/u_0$ and $\bar{c}/u_0$
have $\gcd$ equal to 1.
Hence there exists $t_{0}\in \NN$  such that $\forall t\ge t_{0}$,
$I_{t}= u_{0}\, (u_{0},u_{1})^{t-1}$.
In this case $\dim_\RR (S/I)_t = 1$ for $t$ sufficiently large,
and it follows that the Hilbert polynomial $HP_{S/I}$  
is the constant polynomial equal to 1. 

Then the exact sequence 
\[
	0 \rightarrow I \rightarrow S \rightarrow S/I \rightarrow 0
\]
implies that 
\begin{equation}\label{eq:hp}
	HP_I(t) = HP_S(t) - HP_{S/I}(t) = \binom{t+1}{1} - e,
\end{equation}
where $HP_M$ is the Hilbert polynomial of the module $M$.

By the Graded Hilbert Syzygy Theorem, we get a resolution of the form
\[
	0 \longrightarrow S(- d_1)\oplus \dots \oplus S(- d_L) \xrightarrow{\;\,\lambda\;\,}
	S(-d_{a}) \oplus S(-d_{b}) \oplus S(-d_c) \longrightarrow I \longrightarrow 0.
\]
Notice that this resolution is not necessarily minimal. Since this is an exact sequence, then the Hilbert polynomial of the middle term is the sum of the other two 
Hilbert polynomials, and applying \eqref{eq:hp} we get 
\[
	3t - (d_a + d_b + d_c) + 3 = (t- d_1 + 1) + \cdots + (t - d_L + 1) + (t+1) -e.
\]
It follows that $L = 2$ which proves \emph{(i)}. Furthermore, we have that 
the degrees $d_1$ and $d_2$ of the syzygies satisfy $ d_1+d_2 = d_a + d_b + d_c - e$.

The matrix $\Lambda$ representing $\lambda$ is a $3\times 2$ matrix
\[ 
	\left( \begin{array}{cc}
	\bar{A}_1 & \bar{A}_2 \\
	\bar{B}_1 & \bar{B}_2 \\
	\bar{C}_1 & \bar{C}_2 
	\end{array} \right)
\] 
the first column corresponding to the generator of degree
$d_{1}$ and the second of degree $d_{2}$.
These two syzygies correspond to vectors of polynomial coefficients
of degree $\mu = d_{1} -\min(d_{a},d_{b},d_{c})$ and $\nu = d_2 - \min (d_a,d_b, d_c)$. By Definition \ref{def:Fm}, 
$\min(d_{a},d_{b},d_{c}) = n + \min\bigl(1, F_{\Delta}(\sigma_{1}) ,F_{\Delta}(\sigma_{2}) \bigr)= n+m$, and also $d_a + d_b + d_c = 3n + F_\Delta + 1$.   
Let us assume that $d_1\leq d_2$, then $\mu$ is the smallest degree of the coefficient vector of a syzygy of $(\bar{a}, \bar{b}, \bar{c})$, and $\nu = n-\mu + 1 + F_{\Delta}-e -2\,m$.

\smallskip
 
By exactness, the two columns of $\Lambda$ generate 
$\Syz(\bar{a},\bar{b},\bar{c})$. 
The dehomogenization (by setting $u_0=1$) of the syzygies in $\Syz(\bar{a},\bar{b},\bar{c})$
leads to syzygies of $(a,b,c)$ over $\RR[u_1]$. In particular, it is 
straightforward to show that
the dehomogenization $({A}_{i},{B}_{i},{C}_{i})$ of 
$(\bar{A}_{i},\bar{B}_{i},\bar{C}_{i})$ for $i=1,2$ generate $Z = \Syz(a,b,c)$ as a module over $\RR[u_1]$. This proves \emph{(ii)}.

By Claim \ref{claim:Zk}, the space $Z_{k}$ is in correspondence with the space of homogeneous syzygies 
of degree $n+k$, which is spanned by the multiples of degree $n+k$ of 
$(\bar{A}_{1},\bar{B}_{1},\bar{C}_{1})$ and 
$(\bar{A}_{2},\bar{B}_{2},\bar{C}_{2})$.
Therefore,
\begin{eqnarray*}
	\dim Z_{k} &=& (n+k-d_{1}+1)_{+} + (n+k-d_{2}+1)_{+}\\ 
 	&=& (k-\mu - m +1)_{+} + (k-\nu-m+1)_{+}
\end{eqnarray*}
with $\nu = n-\mu + 1 + F_{\Delta}-e -2\,m$.
This proves \emph{(iii)}.

The point \emph{(iv)} is a consequence of Hilbert-Burch theorem.
More details on this proof can be found in \cite[Chapter 6, \S\,4.17]{usingalggeom}.
\end{proof}
\begin{definition}\label{def:mu}
  For an interior edge $\tau$ in the topological surface $\cM$ shared by the faces
  $\sigma_{1}$, $\sigma_{2}$ with gluing data $[a_{\tau},b_{\tau},c_{\tau}]$, we denote by
	$\cN_{\tau}$ the topological surface formed by the cells $\sigma_{1}$,
	$\sigma_{2}$ glued along the edge $\tau$ with the same gluing data.
	Let $\mu_\tau$ be the smallest coefficient 
	degree among the two generators
	of the module $Z = \Syz(a_{\tau},b_{\tau},c_{\tau})$.
	Let $\nu_{\tau}= n_{\tau}-\mu_{\tau} + 1 + F_{\Delta}(\tau)-e_{\tau} -2\,m_{\tau}$
	denote the complementary degree, where
	$n_{\tau}=\max(\deg(a_{\tau}),\deg(c_{\tau}),\deg(c_{\tau})) $,
	$m_{\tau}=\min(F_{\Delta}(\sigma_{1}), F_{\Delta}(\sigma_{2}))$,
	$e_{\tau}= \min( n_{\tau} + 1-\deg(a_{\tau}), 
 	n_{\tau} + F_{\Delta}(\sigma_{2})-\deg(b_{\tau}), 
 	n_{\tau} + F_{\Delta}(\sigma_{1})-\deg(c_{\tau}))$
	$F_{\Delta}(\tau)= F_{\Delta}(\sigma_{1})+ F_{\Delta}(\sigma_{2})$.
	The corresponding basis of the syzygy module $Z$ of
	$[a_{\tau},b_{\tau},c_{\tau}]$ is called the $\mu_{\tau}$-basis. 
\end{definition}
This construction allows us to determine the dimension of $S^{1}_{k}(\cN_{\tau})$.
\begin{proposition}\label{prop:dimN}
	For $F_{\Box}$ (resp. $F_{\Delta}$) the number of
	rectangles (resp. triangles) of $\cN_{\tau}$, 
	\small
 	\begin{eqnarray*}
		\dim \cS_{k}(\cN_{\tau}) &=& 
		1+ (k^{2}-1) F_{\Box} + \frac{1}{2}(k^{2} - k ) F_{\Delta} + d_{\tau}(k)\\
 		\mbox{where\ \ \ }\ d_{\tau}(k)&=&
 		(k - \mu_{\tau} -m_{\tau} + 1)_+ + (k - n_{\tau} + \mu_{\tau} + m_{\tau} - 
 		F_{\Delta} + e_{\tau})_+
	\end{eqnarray*}
\end{proposition}
\begin{proof}
Since the only constraints satisfied by the spline functions in
$\cS(\cN_{\tau})$ are the gluing conditions along the edge $\tau$,
the number of linearly independent splines on $\cN_{\tau}$ can be easily counted
by using (\ref{eq:edgeint1}) and (\ref{eq:edgeint2}), and the linearly independent
terms in the Bernstein-B\'ezier 
representation of the polynomials $f_1$ and $f_2$ that conform a spline. 

The gluing data and the smoothness along the edge $\tau$ impose conditions on the 
terms in $f_1$ and $f_2$ which are linear in $v_1$ and $v_2$, respectively. 
Thus, the dimension of the space of splines on $\cN_{\tau}$ of degree exactly $1$ in $v_1$ and
$v_2$, is given by $\dim Z_{k} = d_{\tau}(k)$.
The formula for $d_\tau$ follows from Lemma \ref{lm:syzygy}, by 
considering $Z = \Syz (a_{\tau}, b_{\tau}, c_{\tau})$, where $a_{\tau},b_{\tau},c_{\tau}$ define the gluing data
along $\tau$.
\end{proof}

\subsection{Separation of vertices}\label{sec:sepvert}

We analyze now the separability of the spline functions on an edge,
that is when the Taylor map at the vertices separate the spline functions.

Let $f=(f_{1},f_{2}) \in \cR(\sigma_{1})\oplus \cR(\sigma_{2})$ of the
form $f_{i}(u_{i},v_{i})= p_{i}
+ q_{i}\, u_{i} + \tilde{q}_{i}\, v_{i} + s_{i} \, u_{i}v_{i}+ r_{i}
\,u_{i}^{2} + \tilde{r}_{i} \, v_{i}^{2}+ \cdots$. Then
\[ 
	T_{\gamma}(f)=[p_{1}, q_{1},\tilde{q}_{1},s_{1},p_{2},
	q_{2},\tilde{q}_{2},s_{2}].
\]
If $f=(f_{1},f_{2})\in \cS^{1}_{k}(\cN_{\tau})$, then 
taking the Taylor expansion of the gluing condition 
\eqref{eq:edgecond}  centered at $u_{1}=0$ yields
\begin{eqnarray}
	\tilde{q}_{1} + s_{1}\, u_{1} &=& (\ma(0) +\ma'(0) u_{1}+ \cdots)
	\, ( \tilde{q}_{2} + 2\, \tilde{r}_{2}\, u_{1} + \cdots) \label{eq:exptaylor}\\
	&&+ (\mb(0) +\mb'(0) u_{1}+ \cdots) 
	\, ( q_{2} + s_{2}\, u_{1} +\cdots )\nonumber
\end{eqnarray}
Combining \eqref{eq:exptaylor} with \eqref{eq:edgecond0} yields 
\begin{eqnarray*}
	p_{1} & = & p_{2} \\ 
 	q_{1} & = & \tilde{q}_{2} \\
 	r_{1} & = & \tilde{r}_{2} \\
	\tilde{q}_{1} & = & \ma(0) \, \tilde{q}_{2} + \mb(0)\, q_{2}\\
	s_{1} &=& 2\,\ma (0)\, r_{2} + \mb (0) \, s_{2}
 	+ \ma'(0) \, \tilde{q}_{2}+ \mb'(0) \, q_{2}. 
\end{eqnarray*}
Let $\cH(\gamma)$ be the linear space spanned
by the vectors $[p_{1}, q_{1},\tilde{q}_{1},s_{1},p_{2}, q_{2},\tilde{q}_{2},s_{2}]$,
which are solution of these equations.

If $\ma(0)\neq 0$, it is a space of dimension $5$ otherwise its
dimension is $4$. Thus $\dim \cH(\gamma)=5 -\cross_{\tau}(\gamma)$.
\begin{proposition}\label{prop:sepg}
	For $k\ge \nu_{\tau}+m_{\tau}+1$, $T_{\gamma}(\cS^{1}_{k}(\cN_{\tau}))=\cH(\gamma)$. Its
	dimension is $\dim  T_{\gamma}(\cS^{1}_{k}(\cN_{\tau})) = 5 -\cross_{\tau}(\gamma).$ 
\end{proposition}
\begin{proof}
Let $G(\gamma)=T_{\gamma}(\cS^{1}_{k}(\cN_{\tau}))$. By construction $G(\gamma)\subset
\cH(\gamma)$. We are going to prove that for $k\ge \nu_{\tau}+m_{\tau}+1$,
$G(\gamma)$ and $\cH(\gamma)$ have the same dimension and thus are equal.

By the decompositions \eqref{eq:edgeint1} and \eqref{eq:edgeint2},
the elements of $T_{\gamma} (\cS^{1}_{k}(\cN_{\tau}))$ are of the form
\[
	[c_{0},A(0),-C(0),-C'(0),c_{0},B(0),A(0),B'(0)]
\]
where $c_{0}\in \RR$ and $[A,B,C]\in Z_{k}$.
By Lemma \ref{lm:syzygy}, an element of $Z_{k}$ is of the form
$[A,B,C] = P \, [A_{1},B_{1},C_{1}] +
Q\, [A_{2}, B_{2}, C_{2}]$ with $P,Q \in \RR[u]$, $\deg(P)\le k-\mu_{\tau}-m_{\tau}$ and $\deg(Q)\le k-\nu_{\tau}-m_{\tau}$.
By removing the repeated columns, reordering and changing some signs, we see that $G(\gamma)= T_{\gamma}(\cS^{1}_{k}(\cN_{\tau}))$ is
isomorphic to the space spanned by the elements
\begin{equation}\label{eq:ImHg}
	\left[
	\begin{array}{c}
 	f_{1}(\gamma)\\
 	\partial_{u_{1}}f_{1}(\gamma)\\ 
 	\partial_{u_{2}}f_{2}(\gamma)\\ 
 	-\partial_{v_{1}}f_{1}(\gamma)\\
	\partial_{u_{2}} \partial_{v_{2}}f_{2}(\gamma)\\ 
	-\partial_{u_{1}} \partial_{v_{1}}f_{1}(\gamma)\\ 
	\end{array}
	\right]
	=
 	\left[
	\begin{array}{ccccc}
 		1 & 0 &0 &0&0\\
 		0 & A_{1}(0) & A_{2}(0)& 0 &0\\
 		0 & B_{1}(0) & B_{2}(0)& 0 &0\\
 		0 & C_{1}(0) & C_{2}(0)& 0 &0\\
 		0 & B_{1}'(0) & B_{2}'(0) & B_{1}(0) & B_{2}(0)\\
 		0 & C_{1}'(0) & C_{2}'(0) & C_{1}(0) & C_{2}(0)\\
	\end{array}
	\right]
	\left[
	\begin{array}{c}
 		c_{0}\\ 
 		P(0)\\ 
 		Q(0)\\ 
 		P'(0)\\ 
 		Q'(0)\\ 
	\end{array}
	\right]
\end{equation}
for $P,Q \in \RR[u]$ with $\deg(P)\le k-\mu_{\tau}-m_{\tau}$ and $\deg(Q)\le
k-\nu_{\tau}-m_{\tau}$.
Let us assume that
$k\ge \nu_{\tau}+m_{\tau}+1$ so that $k-\mu_{\tau} -m_{\tau}\ge k-\nu_{\tau}-m_{\tau}\ge 1$.

As $A_{1} B_{2}-A_{2} B_{1}=c$ and $A_{1}(0) B_{2}(0)-A_{2}(0)
B_{1}(0)=c(0)\neq 0$, we deduce
 that $[B_{1}(0),B_{2}(0)]\neq [0,0]$ and that $\dim G(\gamma)\ge 4$.

If $\cross_{\tau}(\gamma)=0$, then
$a(0)=B_{1}(0)C_{2}(0)- B_{2}(0)C_{1}(0)\neq 0$ and
$\dim G(\gamma)= 5 = 5-\cross_{\tau}(\gamma) = \dim \cH(\gamma)$.
 
If $\cross_{\tau}(\gamma)=1$, then $a(0)=B_{1}(0)C_{2}(0)-
B_{2}(0)C_{1}(0)=0$ and $\dim G(\gamma)= 4= 5-\cross_{\tau}(\gamma) = \dim \cH(\gamma)$.

In both cases, we have $\dim G(\gamma)= \dim \cH(\gamma)$, which implies that
$G(\gamma)=\cH(\gamma)$. This completes the proof of the proposition.
\end{proof}
If $\gamma'$ is the other end point of $\tau$, we have a Taylor map
for each $\gamma$ and $\gamma'$, that we join together.
Let
\begin{eqnarray}\label{eq:Tgg}
	T_{\gamma,\gamma'}:\cR(\sigma_{1})\oplus
	\cR(\sigma_{2})&\rightarrow& \cR^{\sigma_{1}}(\gamma)\oplus \cR^{\sigma_{2}}(\gamma) \oplus 
	\cR^{\sigma_{1}}(\gamma')\oplus \cR^{\sigma_{2}}(\gamma') \\
	f=(f_{1},f_{2}) & \mapsto & (T_{\gamma}(f), T_{\gamma'}(f)) \nonumber
\end{eqnarray}
and let $G(\tau)=T_{\gamma,\gamma'}(\cS^{1}_{k}(\cN_{\tau}))$.
\begin{proposition}\label{prop:Hgg}
	For $k\ge \nu_{\tau}+m_{\tau}+4$, we have
	$T_{\gamma,\gamma'}(\cS^{1}_{k}(\cN_{\tau}))=(\cH(\gamma)$, $\cH(\gamma'))$ and
	\[
		\dim T_{\gamma,\gamma'}(\cS^{1}_{k}(\cN_{\tau})) = 10 -\cross_{\tau}(\gamma) 
		-\cross_{\tau}(\gamma').
	\]
\end{proposition}
\begin{proof}
By a change of coordinates, we can assume that the
coordinates of $\gamma$ (resp. $\gamma'$) are $(0,0)$
(resp. $(1,0)$) in $\sigma_{1}$ and $(0,0)$
(resp. $(0,1)$) in $\sigma_{2}$.

Similarly to the proof of the previous proposition, $T_{\gamma'}(\cS^{1}_{k}(\cN_{\tau}))$ is
spanned by the vectors
\[
	[c_{0} + \int_{0}^{1}A(u) du, A(1),-C(1),C'(1),c_{0}+\int_{0}^{1}A(u) du,B(1),A(1),B'(1)]
\]
for $c_{0}\in \RR$ and $[A,B,C] = P \, [A_{1},B_{1},C_{1}] +
Q\, [A_{2}, B_{2}, C_{2}] \in Z_{k}$
with $\deg(P)\le k-\mu_{\tau}-m_{\tau}$ and $\deg(Q)\le k-\nu_{\tau}-m_{\tau}$.

For $k\ge \nu_{\tau}+m_{\tau}+4$,  we can find polynomials 
$P =p_{0} (1-3u^{2}+2u^{3})+ p_{1}(u-2u^{2}-u^{3}) + p_{2} u^{2}(1-u)^{2}$,
$Q=q_{0} (1-3u^{2}+2u^{3})+ q_{1}(u-2u^{2}-u^{3}) + q_{2} u^{2}(1-u)^{2}$,
of degree $\le 4$ such that
$P(0)=p_{0}, P'(0)=p_{1}$, $Q(0)=q_{0}$, $Q'(0)=q_{1}$,  
$P(1)=0, P'(1)=0$, $Q(1)=0$, $Q'(1)=0$ and 
$\int_{0}^{1}(P A_{1}+ Q A_{2})(u)du = -c_{0}$.

This implies that $(\cH(\gamma),0) \subset G(\tau)$.
By symmetry, we also have $(0,\cH(\gamma')) \subset G(\tau)$. 
By construction $G(\tau)\subset (\cH(\gamma)$, $\cH(\gamma'))$,
therefore we have $G(\tau)= (\cH(\gamma), \cH(\gamma'))$
and $\dim G(\tau)= \dim \cH(\gamma) + \dim \cH(\gamma')$.
We deduce the dimension formula from Proposition \ref{prop:sepg}.
\end{proof}
\begin{definition} The separability $\ms(\tau)$ of the edge $\tau$ is the minimal
  $k$ such that $T_{\gamma,\gamma'}(\cS^{1}_{k}(\cN_{\tau}))=
	(T_{\gamma}(\cS^{1}_{k}(\cN_{\tau})), T_{\gamma'}(\cS^{1}_{k}(\cN_{\tau})))$.
\end{definition}
\begin{remark} The bound $\nu_{\tau}+m_{\tau}+4\ge \ms(\tau)$ is not necessarily
	the minimal degree of separability. Separability can be attained  as
 	soon as $d_{\tau}(k)\ge 9 - \cross_{\tau}(\gamma) -\cross_{\tau}(\gamma').$
\end{remark}
 
\subsection{Decompositions and dimension} \label{sec:polybasis}
Let $\tau\in\cM_{1}$ be an interior edge $\tau$ shared by the cells
$\sigma_{1}, \sigma_{2}\in \cM_{2}$. Let $K_{1}= (v_{1}^{2})\cap \cR_{k}(\sigma_{1})$
and $K_{2}= (u_{2}^{2})\cap \cR_{k}(\sigma_{2})$
be the polynomials of $\cR _{k}(\sigma_{1})$ (resp. $\cR _{k}(\sigma_{2})$) divisible by
$v_{1}^{2}$ (resp. $u_{2}^{2}$).
Let $L$ be
the subspace of polynomials of  $\cR_{k}({\sigma_{1}})\oplus
\cR_{k}({\sigma_{2}})$ spanned by the Bernstein basis functions on
$\sigma_{1}$ and $\sigma_{2}$, which are not divisible by
$v_{1}^{2}$ or $u_{2}^{2}$ and let  $\pi_{L}$ be the projection of $\cR_{k}({\sigma_{1}})\oplus \cR_{k}({\sigma_{2}})$
on $L$ along $(K_{1},0)\oplus (0,K_{2})$.
The functions in $L$ are said to have {\em their support along} $\tau$.
By construction, we have $\cR_{k}({\sigma_{1}})\oplus \cR_{k}({\sigma_{2}})= (K_{1},0) \oplus 
(0,K_{2})\oplus L$.
The elements of $(K_{1},K_{2})$ are obviously in $\cS_{k}^{1}(\cN_{\tau})$ since they vanish at the order $1$
along $\tau$.

Let $W_{k}(\tau)= \pi_{L}(\Theta_{\tau}(Z_{k}))$ where
$\Theta_{\tau}$ is defined in \eqref{eq:deftheta}. 
Notice that $W_{k}(\tau)\subset \cS_{k}^{1}(\cN_{\tau})$
since $\ker \pi_{L}\subset \cS_{k}^{1}(\cN_{\tau})$. Moreover, 
since $\ker \pi_{L}$ does not intersect $\Theta_{\tau}(Z_{k})$ and
$\Theta_{\tau}$ is injective, the spaces 
$W_{k}(\tau)$, $\Theta_{\tau}(Z_{k})$ and $Z_{k}$ have the same dimension.
Therefore,  we have
$\dim(W_{k}(\tau))=d_{\tau}(k)$ and 
$W_{k}(\tau) \neq \{0\}$ when $k\geq \mu_{\tau} + m$
(Lemma \ref{lm:syzygy} (iii)). 

From the relations \eqref{eq:edgeint1} and \eqref{eq:edgeint2},
we deduce the following decomposition:
\begin{equation}\label{prop:S1N}
	\cS_{k}^{1}(\cN_{\tau}) = (K_{1},0) \oplus (0, K_{2}) \oplus \RR\, \ub \oplus W_{k}(\tau)
\end{equation}
where $\ub=\pi_{L}((1,1))$.
The sum of these spaces is direct, since the supports of the functions
of each space do not intersect.

The map $T_{\gamma,\gamma'}$ defined in \eqref{eq:Tgg} induces the
exact sequence
\[ 
	0 \rightarrow \cK_{k}(\tau) \rightarrow \cS_{k}^{1}(\cN_{\tau}) 
	\stackrel{T_{\gamma,\gamma'}}{\longrightarrow} G(\tau) \rightarrow 0
\]
where $\cK_{k}(\tau)= \ker T_{\gamma,\gamma'}$
and $G(\tau)=T_{\gamma,\gamma'}(\cS_{k}^{1}(\cN_{\tau}))$.
It is clear that $(K_{1},K_{2})\subset \cK_{k}(\tau)$.
\begin{definition}\label{def:edgespline}
	For an interior edge $\tau\in \cM_{1}^{o}$, let 
	$\cE_{k}(\tau)=\ker (T_{\gamma,\gamma'})\cap W_{k}(\tau)$ be the
	set  of splines in $\cS_{k}^{1}(\cN_{\tau})$ with their 
	support along $\tau$ and with vanishing 
	Taylor expansions at $\gamma$ and $\gamma'$.
	For a boundary edge $\tau'=(\gamma,\gamma')$, which belongs to a face $\sigma$, we 
	also define $\cE_{k}(\tau')$ as the set of elements of $\cR_{k}(\sigma)$
	with their  support along $\tau'$ and with vanishing 
	Taylor expansions at $\gamma$ and $\gamma'$.
\end{definition}
Notice that the elements of $\cE_{k}(\tau)$ have their support along
$\tau$ and their Taylor expansion at $\gamma$ and $\gamma'$
vanish. Therefore, their Taylor expansion along all (boundary) edges of
$\cN_{\tau}$ distinct from $\tau$ also vanish.
\begin{lemma}\label{lm:Ktau}
	For an interior edge $\tau\in \cM_{1}^{o}$, we have $\cK_{k}(\tau) = (K_{1},0) \oplus (0, K_{2}) 	\oplus \cE_{k}(\tau)$.
\end{lemma}
\begin{proof}
As $(K_{1},0), (0,K_{2}) \subset \ker T_{\gamma,\gamma'}=\cK_{k}(\tau)$
and $\cK_{k}(\tau)\cap \left(W_{k}(\tau) \oplus \RR \,\ub\right)= \cK_{k}(\tau) \cap
W_{k}(\tau) =\cE_{k}(\tau)$, we have
\begin{align*}
	\cK_{k}(\tau) &= (K_{1},0) \oplus (0, K_{2}) \oplus
	\left((W_{k}(\tau) \oplus \RR\, \ub) \cap \cK_{k}(\tau)\right)\\
	&= (K_{1},0) \oplus (0, K_{2}) \oplus \cE_{k}(\tau). \qedhere
\end{align*}
\end{proof}
\begin{corollary}
	For an interior edge $\tau\in \cM_{1}^{o}$ and for $k\ge \ms(\tau)$, 
	the dimension of $\cE_{k}(\tau)$ is
	\[ 
		\dim \cE_{k}(\tau) = d_{\tau}(k) -9 + \cross_{\tau}(\gamma) +\cross_{\tau}(\gamma'). 
	\]
\end{corollary}
\begin{proof}
By Lemma \ref{lm:Ktau}, we have
\[
	\dim \cE_{k}(\tau) = \dim \cK_{k}(\tau) - \dim K_{1} -\dim K_{2}.
\]
As $\cK_{k}(\tau)$ is the kernel of $T_{\gamma,\gamma'}$ and
$G(\tau)$ is its image, we have
\[
	\dim \cK_{k}(\tau) = \dim S^{1}_{k}(\cN_{\tau}) -\dim G(\tau).
\]
As $\dim(W_{k}(\tau))= d_{\tau}({k})$, we deduce from the
decomposition \eqref{prop:S1N} that
$\dim S^{1}_{k}(\cN_{\tau}) = 1+ d_{\tau}(k)+ \dim K_{1} +
\dim K_{2}$.
Using Proposition \ref{prop:Hgg}, $G(\tau)=
(\cH(\gamma),\cH(\gamma'))$ and we obtain
\begin{align*}
	\dim \cE_{k}(\tau)&=\dim S^{1}_{k}(\cN_{\tau}) -\dim G(\tau) -
                      \dim K_{1} -\dim K_{2}\\
  & = d_{\tau}(k) -9+ \cross_{\tau}(\gamma) +\cross_{\tau}(\gamma'). \qedhere
\end{align*}
\end{proof}
\begin{remark}\label{rem:boundary:edge}
	When $\tau$ is a boundary edge, which belongs to the face $\sigma_{1}\in
	\cM_{2}$, we have
	$\cS_{k}(\cN_{\tau})=\cR_{k}(\sigma_{1})$, $\cK_{k}(\tau)=
	K_{1} \oplus \cE_{k}(\tau)$ 
	and for $k\geq \ms(\tau)= 3+F_{\Delta}(\sigma_{1})$,
	$\dim G(\tau)=8$ and $\dim \cE_{k}(\tau) = k+ 1 +(k+1 -
	F_{\Delta}(\sigma_{1}))  -8= 2 \, k  - F_{\Delta}(\sigma_{1}) - 6$.

	Notice that this is also what we obtain if we attach a virtually 
	rectangular face along $\tau$ with constant gluing data:
	$n=\mu=0$, $m=0$, $e=2$, $\cross_{\tau}(\gamma)=\cross_{\tau}(\gamma')=0$
	and
	\[ 
		d_{\tau}(k) = 2\,k + 3 -  F_{\Delta}(\sigma_{1}),
	\]
	so that $\dim \cE_{k}(\tau)  =  d_{\tau}(k) - 9 +\cross_{\tau}(\gamma) +\cross_{\tau}(\gamma')$.
\end{remark}
 
\section{$G^1$ splines around a vertex}\label{sec:aroundv}

We consider now a topological surface $\cO$ composed of faces
$\sigma_{1},\ldots, \sigma_{F} \in \cO_{2}$ sharing a single vertex
$\gamma$, and such that $\sigma_{i}$ and $\sigma_{i+1}$ share the edge
$\tau_{i+1}=(\gamma, \delta_{i+1})$.
In particular $\tau_{i}, \tau_{i+1}$ are the two edges of $\sigma_{i}$
containing the vertex $\gamma$. 
The number of edges containing $\gamma$ is denoted $F'$.  
All the vertices of $\cO$ different from $\gamma$ are boundary vertices.
The vertex $\gamma$ is an interior vertex, iff $\sigma_{F}$ and
$\sigma_{1}$ share the edge $\tau_{1}$.
In this case, we identify
the indices modulo $F$ and we have $F'=F$, otherwise we have $F'=F+1$.
The gluing data for the interior edge $\tau_{i}$ is
$\ma_{i}=\frac{a_{i}}{c_{i}}, \mb_{i}=\frac{b_{i}}{c_{i}}$.

The coordinates in the ring $\cR(\sigma_{i})$ are chosen so that
the coordinates of $\gamma$ 
are $(0,0)$ and $\tau_{i}$ is defined by $v_{i}=0$, $u_{i}\in [0,1]$
and by $u_{i-1}=0$, $v_{i-1}\in [0,1]$ in $\cR(\sigma_{i-1})$.
The canonical form of the  transition map at $\gamma$ across the edge $\tau_{i}$ is then 
\[
	\phi_{\tau_{i}}: (u,v) \longrightarrow 
	\begin{pmatrix}
		v_{i} \,\mb_{i}(u_{i})\\
		u_{i}+v_{i}\,\ma_{i}(u_{i}) 
	\end{pmatrix}
\]
Let $f=(f_{i})_{i=1,\ldots,F}\in \cS^{1}(\cO)$.
The gluing condition \eqref{eq:edgecond} implies that the Taylor expansion of
$f_{i}$ at $\gamma$ is of the form
\[
	f_{i}(u_{i},v_{i})= p + q_{i}\, u_{i} + q_{i+1} \,v_{i} + s_{i} \,
	u_{i}v_{i} + r_{i} \, u_{i}^{2}+ r_{i+1} \, v_{i}^{2} + \cdots
\]
for $p,q_{i}, s_{i}, r_{i} \in \RR$, $i=1, \ldots, F$ (see Fig. \ref{fig:taylor}).
\begin{figure}[ht]
	\begin{center}
	\includegraphics[width=5.2cm]{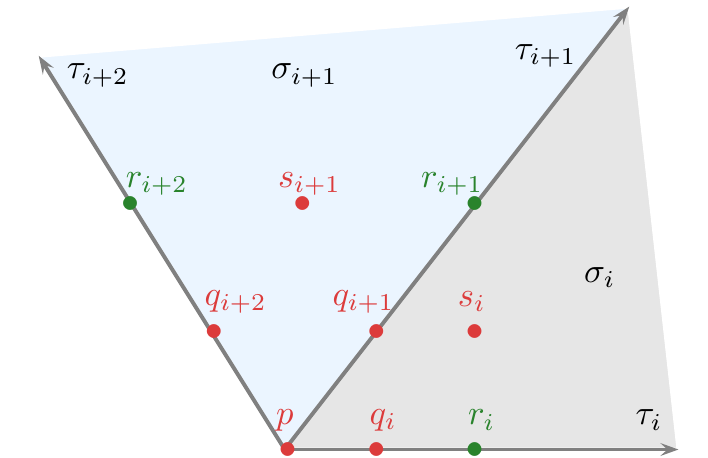}
	\end{center}
	\vspace{-0.5cm}
	\caption{Taylor coefficients around a vertex.}\label{fig:taylor}
\end{figure}

\noindent By a computation similar to \eqref{eq:exptaylor}, Condition
\eqref{eq:edgecond} implies that
\begin{eqnarray}
	q_{i+1} & = & \ma_{i}(0)\, q_{i}+ \mb_{i}(0)\, q_{i-1} \ \
 	\hspace{3.5cm} \ \ (i=2,\ldots,F)\ \ \label{eq:recg1.1}\\
	 s_{i} & = & 2\,\ma_{i}(0)\, r_{i} +
 	\mb_{i}(0) \, s_{i-1} + \ma_{i}'(0) \, q_{i}+
 	\mb_{i}'(0) \, q_{i-1} \ \ \ (i=2,\ldots,F)\ \ \label{eq:recg1.2}
\end{eqnarray}
Let $\cH(\gamma)$ be the vector space spanned by the vectors~$\bh=[p,q_{1},\ldots,q_{F'},s_{1},\ldots,s_{F}]$
for~$\bh'=[p,q_{1},\ldots,q_{F'},s_{1},\ldots,s_{F},r_{1},\ldots,r_{F'}]$
a solution of  the linear system \eqref{eq:recg1.1}, \eqref{eq:recg1.2}.
\begin{proposition} \label{prop:Hg}
	\[ 
		\dim \cH(\gamma) = 3 + F(\gamma) -\sum_{\tau \ni \gamma}
		\cross_{\tau}(\gamma) + \cross_{+}(\gamma)
	\]
	where $F=F(\gamma)$ is the number of faces around the vertex $\gamma$.
\end{proposition}
\begin{proof}
Notice that $\cH(\gamma)$ is isomorphic to the projection of the solution set of system 
\eqref{eq:recg1.1}, \eqref{eq:recg1.2} on 
the space of the variables $[p, \bq,\bs]=[p,q_{1},\ldots,q_{F'},s_{1},\ldots,s_{F}]$.
 
The solutions in $\bq=(q_{1},\ldots,q_{F'})$ of the first set of
equations satisfy the induction relations
\[
	\left(\begin{array}{c} q_{i} \\ q_{i+1} \end{array} \right) =
	\left(
 	\begin{array}{cc} 0 & 1 \\
 	\mb_i(0) & \ma_i(0) \end{array}
	\right) \, \left(\begin{array}{c} q_{i-1} \\ q_{i} \end{array} \right)
	\ \mathrm{for}\ i=2,\ldots,F.
\]
As we have the compatibility condition \ref{cond:comp1} at an interior vertex,
the solutions of \eqref{eq:recg1.1} span a linear
space of dimension $2$, parametrized for instance by $q_{1},q_{2}$.

The system \eqref{eq:recg1.2} is formed by linearly independent
equations which involve $r_{k}$ and $q_{i},s_{j}$ when $\ma_{k}(0)\neq 0$
and by equations which only involve $s_{i} ,s_{i-1}$ and $q_{j}$ when $\ma_{i}(0)= 0$.

Therefore the projection of the solution set of \eqref{eq:recg1.2} on the space corresponding to the variables $[p, \bq,\bs]$ 
is defined by the equations which only involve $s_{i} ,s_{i-1}$ and
$q_{i}, q_{i-1}$ when $\ma_{i}(0)= 0$.

If one of the edges around $\gamma$ is not a crossing edge, then the
codimension of this space is
$\sum_{\tau \ni \gamma} \cross_{\tau}(\gamma)$, with the convention
that $\cross_{\tau}(\gamma)=0$
if $\tau$ is a boundary edge. 

If all the edges around $\gamma$ are crossing edges, then the
compatibility conditions \ref{cond:comp2} at $\gamma$ imply that
one of these equations is dependent from the other.
Therefore, the codimension of this space is $\sum_{\tau \ni \gamma}
\cross_{\tau}(\gamma) -\cross_{+}(\gamma)$.

By intersecting it with the solution space of \eqref{eq:recg1.1}, we
deduce that the dimension of $\cH(\gamma)$ is precisely given by
\[
	1 + 2 + F -\sum_{\tau \ni \gamma} \cross_{\tau}(\gamma) + \cross_{+}(\gamma).\qedhere
\]
\end{proof}

Let $T_{\gamma}=\prod_{\sigma\ni \gamma} T_{\gamma}^{\sigma}$ be the
Taylor map at $\gamma$ on $\cO$
and let $T_{\partial\cO}= \prod_{\tau\not\ni\gamma} T_{\tau}^{\sigma}$ be
the Taylor map along all the boundary edges which do not contain $\gamma$.

For $k\in \NN$, we define $\cV_{k}(\gamma)= \ker T_{\partial\cO} \cap \cS_{k}^{1}(\cO)$ the
set of $G^{1}$ spline functions on $\cO$ which vanish at the first
order along the boundary edges (which do not contain $\gamma$). 
\begin{proposition}\label{prop:HgT}
	For $k\ge \max_{i=1,\ldots,F}(\ms(\tau_{i}))$,
	$T_{\gamma}(\cV_{k}(\gamma))= \cH(\gamma)$.
\end{proposition}
\begin{proof}
By construction, the elements of $\cV_{k}(\gamma)$ satisfy the equations
\eqref{eq:recg1.1}, \eqref{eq:recg1.2}.
This implies that $T_{\gamma}(\cV_{k}(\gamma))\subset \cH(\gamma)$.

Consider an element $\bh= (h_{1},\ldots, h_{F}) \in \cH(\gamma)$. By Proposition \ref{prop:Hgg},
for $k\ge \ms(\tau_{i})$, there exists
$(f_{i},\tilde{f}_{i})\in \cS_{k}^{1}(\cN_{\tau_{i}})$ 
such that $T_{\gamma} (f_{i},\tilde{f}_{i}) =(h_{i},h_{i-1})$ and $T_{\delta_{i}} (f_{i},\tilde{f}_{i}) =0$.
Let $v_{i}=0$ (resp. $u_{i-1}=0$) be the equation of $\tau_{i}$ in
$\sigma_{i}$ (resp. $\sigma_{i-1}$). As for any polynomials
$p\in (v_{i}^{2}) \cap \cR_{k}(\sigma_{i}), q \in (u_{i-1})^{2} \cap
\cR_{k}(\sigma_{i-1})$,
$T_{\gamma}(p,q)=0$, 
we can assume that $(f_{i},\tilde{f}_{i})$ has its support in $\cR^{\sigma_{i}}(\tau_{i}) \oplus \cR^{\sigma_{i-1}}(\tau_{i})$.

By construction, we have $T_{\gamma}^{\sigma_{i}}(f_{i}) =
T_{\gamma}^{\sigma_{i}}(\tilde{f}_{i+1})= h_{i}$. Thus, there exist
$g_{i}\in \cR_{k}(\sigma_{i})$ supported in $\cR^{\sigma_{i}}(\tau_{i}) + \cR^{\sigma_{i}}(\tau_{i+1}) $
such that $T_{\tau_{i}}^{\sigma_{i}}(g_{i})=f_{i}$,
$T_{\tau_{i+1}}^{\sigma_{i}}(g_{i}) =\tilde{f}_{i+1}$. It is constructed
by taking the coefficients of $f_{i}$ on $\cR^{\sigma_{i}}(\tau_{i})$
and those of $\tilde{f}_{i+1}$ on $\cR^{\sigma_{i}}(\tau_{i+1})$, the
coefficients in $\cR^{\sigma_{i}}(\tau_{i}) \cap \cR^{\sigma_{i}}(\tau_{i+1}) $
coinciding (see Fig. \ref{fig:edgefunction}). As
$T_{\delta_{i}}^{\sigma_{i}}(f_{i}) =
T_{\delta_{i}}^{\sigma_{i}}(g_{i}) =0$,
$T_{\delta_{i+1}}^{\sigma_{i}}(\tilde{f}_{i})=
T_{\delta_{i+1}}^{\sigma_{i}}(g_{i})= 0$
and $g_{i}$ is supported in $\cR^{\sigma_{i}}(\tau_{i}) + \cR^{\sigma_{i}}(\tau_{i+1})$,
we have $T_{\tau}^{\sigma_{i}}(g_{i})=0$ for any edge $\tau$ of the
face $\sigma_{i}$, which does not contain $\gamma$.
\begin{figure}[ht]
	\begin{center}
	\includegraphics[height=4.8cm]{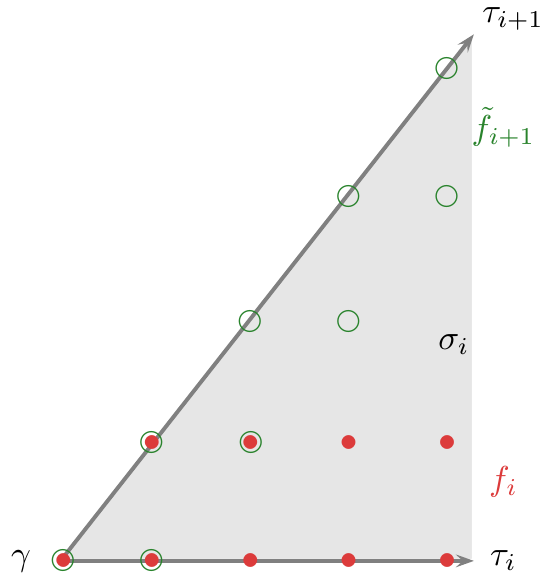}
	\end{center} 
 	\vspace{-0.5cm}
	\caption{Lifting the edge functions.}\label{fig:edgefunction}
\end{figure}

Let $\bg=[g_{1},\ldots,g_{F}]\in \oplus_{\sigma_{i}\ni \gamma}
\cR_{k}(\sigma_{i})$. By construction, $\bg$ vanishes at the first order along all the
boundary edges of $\cO$, which do not contain $\gamma$.
Moreover, $T_{\tau_{i}}(\bg)=(f_{i},\tilde{f}_{i}) \in
\cS_{k}^{1}(\cN_{\tau_{i}})$, thus $\bg$ satisfies the gluing
conditions along the edge $\tau_{i}$.
We also have $T_{\tau}(\bg)=0$ for any edge $\tau$, which does not
contain $\gamma$.
Thus $\bg$ satisfies the gluing conditions along all the edges and its
image by $T_{\partial \cO}$ vanishes, i.e. 
$\bg \in \cS^{1}_{k}(\cO) \cap \ker T_{\partial \cO}= \cV_{k}(\gamma)$. By construction,
$T_{\gamma}(\bg)=\bh$. This shows that $\cH(\gamma)\subset T_{\gamma}(\cV_{k}(\gamma))$
and concludes the proof.
\end{proof}

\section{$G^1$ splines on a general mesh}
We consider now a general mesh $\cM$ with an arbitrary number of
faces, possibly with boundary edges.

We denote by $T_{0}=\prod_{\gamma\in \cM_{0}} T_{\gamma}$ 
the Taylor map at all the vertices of $\cM$ and $\cH= T_{0}(S_{k}^{1}(\cM))$.
We have the following exact sequence:
\[ 
0 \rightarrow \cK_{k} \rightarrow \cS_{k}^{1}(\cM) \stackrel{T_{0}}{\rightarrow} G \rightarrow 0 
\]
where $\cK_{k}=\ker T_{0} \cap \cS_{k}^{1}(\cM)$
and $G=T_{0}(\cS_{k}^{1}(\cM))$. 
Let $\ms^{*} = \max \{\ms(\tau)\mid \tau\in \cM_{1}\}$. We have
$\ms^{*}\leq \max \{\nu_{\tau}+m_{\tau}+4\mid \tau\in \cM_{1}\}$

\subsection{Splines at a vertex}

Let $\gamma\in \cM_{0}$ be a vertex of $\cM$ and
let $\cO_{\gamma}$ be the sub-mesh associated to
the faces of $\cM$ which contain $\gamma$.
Let $\cV_{k}(\gamma)$ be the set of spline functions in $\cS_{k}^{1}(\cM)$
supported on the faces of $\cO_{\gamma}$, which vanish at the first
order along the edges that do not contain $\gamma$.
\begin{proposition}\label{prop:H}
	For $k\ge \ms^{*}$, $T_{0}(\cS_{k}^{1}(\cM)) =\prod_{\gamma} \cH(\gamma)$ and
 	\[ 
 		\dim T_{0}(\cS_{k}^{1}(\cM))  = \sum_{\gamma\in\cM_{0}} (F(\gamma)+3) 
 		-\sum_{\gamma\in \cM_{0}}\sum_{\tau \ni \gamma}
 		\cross_{\tau}(\gamma) + \sum_{\gamma\in \cM_{0}} \cross_{+}(\gamma),
 	\]
where $F(\gamma)$ is the number of faces of $\cM$ that contain the vertex
$\gamma\in \cM_{0}$.
\end{proposition}
\begin{proof}
By proposition \ref{prop:HgT}, for $k\ge \ms^{*}$ the image of
$\cV_{k}(\gamma)$ by $T_{0}$ is $G(\gamma)$,
and $T_{\gamma'}(\cV_{k}(\gamma))=0$ for any other vertex $\gamma'\neq
\gamma$.

This shows that $T_{0}(\cS_{k}^{1}(\cM) )
= \prod_{\gamma} G(\gamma)
= \prod_{\gamma} \cH(\gamma)$.
We deduce the dimension formula from Proposition \ref{prop:Hg}.
\end{proof}

\subsection{Splines on edges}

For an interior edge $\tau=(\gamma,\gamma')\in \cM_{1}$, let $\cN_{\tau}$ be
the sub-mesh made of the faces $\sigma_{1},\sigma_{2}$ of $\cM$ containing $\tau$.
Let $\cE_{k}(\tau)=\ker T_{\gamma,\gamma'}\cap \Theta_{\tau}(Z_{k})$
(see Definition \ref{def:edgespline}).
The elements of $\cE_{k}(\tau)$ correspond to splines of
$\cS_{k}^{1}(\cN_{\tau})$,
which are in the kernel of $T_{\gamma,\gamma'}$ and with a support in
$\cR^{\sigma_{1}}(\tau)\oplus \cR^{\sigma_{2}}(\tau)$.
Thus, $\cE_{k}(\tau)\subset \ker T_{\tau'}$ for any edge $\tau'\in
\cM_{1}$, distinct from $\tau$. We deduce that any element of
$\cE_{k}(\tau)$ satisfies the gluing condition along all edges of
$\cM_{1}$, and thus corresponds to a spline function in
$\cS_{k}^{1}(\cM)$.
In other words, we have $\cE_{k}(\tau)\subset \cS_{k}^{1}(\cM)\cap \ker T_{0}=\cK_{k}$.
The elements of $\cE_{k}(\tau)$ have a support in $\cR^{\sigma_{1}}(\tau)\oplus \cR^{\sigma_{2}}(\tau)$
and their Taylor coefficients at the end points of $\tau$ vanish.

Thus the support of the elements of $\cE_{k}(\tau)$ 
and
$\cE_{k}(\tau')$ for two distinct edges $\tau,\tau'$ do not intersect,
and their sum is direct. Let $\cE_{k}=\oplus_{\tau\in \cM_{1}} \cE_{k}(\tau)$.

Let  $\cF_{k}= \ker T_{1} \cap S_{k}^{1}(\cM)$ be the set of spline functions, which Taylor
expansions along all edges vanish. 
\begin{proposition}\label{prop:Kk}
\[
	\cK_{k} = \cF_{k} \oplus \cE_{k}
\]
\end{proposition}
\begin{proof}
Let $f\in \cK_{k}$ and take an interior edge $\tau\in \cM_{1}$. Let
$\sigma_{1}, \sigma_{2}$ be the two faces of $\cN_{\tau}$.

Then $(f_{\sigma_{1}}, f_{\sigma_{2}})\in \cS^{1}_{k}(\cN_{\tau})\cap \ker T_{0}=\cK_{k}(\tau)$.
By Lemma \ref{lm:Ktau}, $(f_{\sigma_{1}}, f_{\sigma_{2}})=
s_{\tau} + (k_{1},k_{2})$ with $s_{\tau}\in \cE_{k}(\tau)$ and 
$(k_{1},k_{2})\in (K_{1},K_{2})$.
As $s_{\tau}$ lifts to a spline $\in S_{k}^{1}(\cM)$, $f-s_{\tau}$ 
is an element of $S_{k}^{1}(\cM)$, which image by the
Taylor expansion $T_{\tau}$ along the edge $\tau$ vanishes.

If $\tau$ is a boundary edge of $\cM$, which belongs to the face
$\sigma_{1}$, we have a similar decomposition $f_{1}=s_{\tau} +
k_{1}$ with $s_{\tau}\in \cE_{k}(\tau)$ and $k_{1}\in K_{1}$, using
 the convention of Remark \ref{rem:boundary:edge}.
Similarly $s_{\tau}$ lifts to a spline $\in S_{k}^{1}(\cM)$, $f-s_{\tau}$ 
is an element of $S_{k}^{1}(\cM)$ in the kernel of $T_{\tau}$.

Repeating this process for all edges $\tau \in\cM_{1}$, we can
construct an element $\tilde{f}=f-\sum_{\tau\in \cM_{1}} s_{\tau}$ such that
$\forall \tau\in \cM_{1}, T_{\tau}(\tilde{f})=0$, i.e. $\tilde{f}$ belongs to 
$\ker T_{1} = \cF_{k}$.
This shows that $\cK_{k} \subset \cF_{k} + \sum_{\tau\in \cM_{1}}
\cE_{k}(\tau)$. By construction, we have $\cF_{k}\subset \cK_{k}$ and
$\cE_{k}=\oplus_{\tau\in \cM_{1}}\cE_{k}(\tau)\subset \cK_{k}$.
Considering the support of the functions in $\cF_{k}$ and
$\cE_{k}$, we deduce that their sum is direct and equal to $\cK_{k}$.
\end{proof}

\subsection{The dimension formula}
We can now determine the dimension of $\cS_{k}^{1} (\cM)$.
\begin{theorem}\label{thm:dim}
	Let $\ms^{*} = \max \{\ms(\tau)\mid \tau\in
	\cM_{1}\}$. Then, for $k\ge \ms^{*},$ 
	\[ 
		\begin{array}{rclll}
 			\dim \cS_{k}^{1} (\cM) & = & (k-3)^{2} F_{\Box} + \frac{1}{2}(k-5)(k-4) F_{\Delta} \\
			&&+ \sum_{\tau\in \cM_{1}} d_{\tau}(k) 
			+ 4 F_{\Box} + 3 F_{\Delta} -9 F_{1} + 3 F_{0} + F_{+} 
		\end{array}
	\]
	where
	\begin{itemize}
  	\item $d_{\tau}(k)$ is the dimension of the syzygies of the gluing
    	data along $\tau$ in degree $\le k$,
   	\item $F_\Box$ is the number of rectangular faces, $F_{\Delta}$
    	is the number of triangular faces,    
  	\item $F_1$ is the number of edges,    
  	\item $F_0$ (resp. $F_{+}$) is the number of (resp. crossing) vertices,
	\end{itemize}
\end{theorem}
\begin{proof}
By definition, we have
\[ 
	\dim \cS_{k}^{1} (\cM) =\dim \cH + \dim \cK_{k}.
\]
By Proposition \ref{prop:Kk}, we have
\begin{eqnarray*}
	\dim \cK_{k} &=& \dim\, \cF_{k} + \dim\, \cE_{k}= \dim \cF_{k} + 
	\sum_{\tau\in \cM_{1}} \dim \cE_{k}(\tau) \\
	& = &(k-3)^{2} F_{\Box}+ \frac{1}{2}(k-5)(k-4) F_{\Delta} +
	\sum_{\tau \in \cM_{1}} (d_{\tau}(k)-9+\cross_{\tau}(\gamma) + \cross_{\tau'}(\gamma'))
\end{eqnarray*}
From Proposition \ref{prop:H},  we deduce that
\begin{eqnarray*}
	\dim \cS_{k}^{1} (\cM) & = &\dim\, \cK_{k}+\dim\, \cH \\
 	& = & (k-3)^{2} F_{\Box}+\frac{1}{2}(k-5)(k-4) F_{\Delta} \\
 	&&+
 	\sum_{\tau=(\gamma,\gamma') \in \cM_{1}} (d_{\tau}(k)-9 + \cross_{\tau}(\gamma) + 
 	\cross_{\tau'}(\gamma'))\\
	&&+ \sum_{\gamma\in\cM_{0}} (F(\gamma)+3) -\sum_{\gamma\in \cM_{0}}\sum_{\tau \ni \gamma}
 	\cross_{\tau}(\gamma) + \sum_{\gamma\in \cM_{0}} \cross_{+}(\gamma)\\
 	&=& F_{\Box} (k-3)^{2} + F_{\Delta} \frac{1}{2}(k-5)(k-4) +
 	\sum_{\tau \in \cM_{1}} d_{\tau}(k) -9 F_{1}\\
 	&& + 4 F_{\Box} + 3 F_{\Delta} + 3 F_{0} + F_{+}
\end{eqnarray*}
since $\sum_{\tau=(\gamma,\gamma')\in \cM_{1}} (\cross_{\tau}(\gamma)
+ \cross_{\tau'}(\gamma'))= \sum_{\gamma\in \cM_{0}}\sum_{\tau \ni
 \gamma} \cross_{\tau}(\gamma) $
and $\sum_{\gamma\in\cM_{0}} F(\gamma)= 4 F_{\Box} + 3 F_{\Delta}$.
\end{proof}

As a direct corollary, we obtain the following result:
\begin{corollary} 
	If $\cM$ is a topological surface with gluing data satisfying the
  compatibility Conditions \ref{cond:comp1}-\ref{cond:comp2} and if all 
  crossing vertices of $\cM$ have 4 edges, then  $\cS_{k}^{1} (\cM)$ is an ample space
  of differentiable functions on $\cM$ for $k\ge \ms^{*}$.
\end{corollary}
 
\subsection{Basis}\label{sec:basis}

We are going now to describe an explicit construction of spline functions 
which form a basis of $\cS_{k}^{1}(\cM)$. 
An algorithmic description of the computation of the Bernstein
coefficients of these basis functions is provided in Appendix A.

We assume that $k$ is bigger than the separability $\ms^{*}$ of all edges.

\subsubsection{Basis functions associated to a vertex}
Let $\gamma\in \cM_{0}$ be a vertex and $\sigma_{1},
\ldots,\sigma_{F}$ be the faces of $\cM_{2}$ adjacent to $\gamma$.
We also assume that $\sigma_{i}$ and $\sigma_{i-1}$ share the edge
$\tau_{i}\in \cM_{1}$ and that $\tau_{1}$ is not a crossing edge at
$\gamma$ if such an edge exists.

To compute the basis functions attached to $\gamma$, we compute 
first the Taylor coefficients of $f_{\sigma_{i}} = p + q_{i} u_{i} + q_{i+1} v_{i} +
s_{i} u_{i}v_{i}+ \cdots$ solutions of the system \eqref{eq:recg1.1}-\eqref{eq:recg1.2}
and then lift these Taylor coefficients to define a spline function
with support in $\cO_{\gamma}$.
This leads to the following type of basis functions:
\begin{itemize}
	\item $1$ basis function attached to the value at $\gamma$: $p=1$,
  		$q_{i}=0$, $s_{i}=0$
 	\item $2$ basis functions attached to the derivatives at $\gamma$:
  		$p=0$, $[q_{1},q_{2}] \in \{[1,0],[0,1]\}$ and $s_{i}=0$ if $\tau_{i}$ is
   		not a crossing edge at $\gamma$ and determined by the relations
   		\eqref{eq:recg1.1}-\eqref{eq:recg1.2}
			if $\tau_{i}$ is a crossing edge at $\gamma$.
 	\item $F(\gamma)-\sum_{i=1}^{F'}\cross_{\tau_{i}}(\gamma)+ \cross_{+}(\gamma)$ basis
    	functions attached to the free cross derivatives, with 
      $p=0$, $q_{i}=0$ and $s_{i}\in \{0,1\}$ if $\tau_{i}$ is not a
      crossing edge and determined by the relations
   		\eqref{eq:recg1.1}-\eqref{eq:recg1.2}
			if $\tau_{i}$ is a crossing edge at $\gamma$.
\end{itemize}

\subsubsection{Basis functions associated to an edge}

Let $\tau$ be an edge of $\cM_{1}$ shared by two faces
$\sigma_{1}$, $\sigma_{2}$ with vertices $\gamma, \gamma'$.
Let us assume that the coordinates of these points in the
face $\sigma_{1}$ are $\gamma=(0,0)$ and $\gamma'=(1,0)$.
 
The elements of $\cE_{k}(\tau)$ are the image by $\Theta_{\tau}$ of the
elements of $Z_{k}$ of the form
$P\,[A_{1},B_{1},C_{1}]+ Q\,[A_{2},B_{2},C_{2}]$ with degree
$\deg(P)\le k-\mu_{\tau}-m_{\tau}$,
$\deg(Q)\le k-\nu_{\tau}-m_{\tau}$
which are in the kernel of $T_{\gamma}$ and $T_{\gamma'}$.

From the relation \eqref{eq:ImHg}, we deduce that $P(0)=0$,
$Q(0)=0$.
That is, $P$ and $Q$ are divisible by $u$.
\begin{itemize}
	\item If $\cross_{\tau}(\gamma)=0$, i.e. $\gamma$ is not a crossing vertex,
		we have $B_{1}(0)C_{2}(0)-B_{2}(0) C_{1}(0)=a(0)\neq 0$ and
 		the relation \eqref{eq:ImHg} implies that $P'(0)=0$,
 		$Q'(0)=0$. That is $P=u^{2} \tilde{P}$, $Q=u^{2}\tilde{Q}$.
 	\item If $\cross_{\tau}(\gamma)=1$, then the kernel of $T_{\gamma}$ is
		generated by polynomials such that $P(0)=0$,
		$Q(0)=0$, $P'(0)=\lambda C_{2}(0)$, $Q'(0)=-\lambda C_{1}(0)$.
		That is $P=u\,(\lambda C_{2}(0)+ u \tilde{P})$, $Q=u\, (-\lambda
		C_{1}(0) + u \, \tilde{Q})$.
\end{itemize}
That is
\[
	P=u\, \left(\lambda \,\cross_{\tau}(\gamma)\, C_{2}(0) + u \,\tilde{P} \right),
	Q=u\, \left(-\lambda\, \cross_{\tau}(\gamma)\, C_{1}(0) + u \,\tilde{Q}\right).
\]
By symmetry at $\gamma'$, we see that $P$ and $Q$ are of the form:
\begin{eqnarray*}
	P&=&u\,(1-u) \left(\lambda \,\cross_{\tau}(\gamma)C_{2}(0)\, (1-u)+ 
	\lambda'\, \cross_{\tau}(\gamma')\,C_{2}(1) \,u+ u\,(1-u)\,\tilde{P}\right),\\
	Q&=&-u\,(1-u) \left(\lambda \,\cross_{\tau}(\gamma)\,C_{1}(0)\, (1-u)+ 
	\lambda' \,	\cross_{\tau}(\gamma')\, C_{1}(1)\, u+ u\,(1-u)\,\tilde{Q}\right), 
\end{eqnarray*}
with $\lambda,\lambda'\in \RR$, $\deg(\tilde{P})\le k-\mu-m-4$,
$\deg(\tilde{Q})\le k-\nu-m-4$.

We construct a basis of $\cE_{k}(\tau)$ by taking the 
image by $\Theta_{\tau}$ of a maximal set of linearly independent
elements of this form (see Section \ref{sec:polybasis}).
This yields $d_{\tau}(k)-9 + \cross_{\tau}(\gamma)+\cross_{\tau}(\gamma')$ spline basis functions.

\subsubsection{Basis functions associated to a face}

Finally, we define the basis functions attached to a face $\sigma\in
\cM_{2}$ as the 2-interior Bernstein basis functions in degree $\le k$.
There are $(k-3)^{2}$ such basis spline functions for a rectangular face and
$k-4 \choose 2$ for a triangular face.

\section{Examples}

\subsection{Splines on flat triangular tilings}
We consider a subdivision of a planar domain $\Omega \subset \RR^{2}$
into a partition of triangles and the topological surface $\cM$ induced by this subdivision.

For two faces $\sigma_{1}, \sigma_{2}\in \cM_{2}$, which share an edge
$\tau \in \cM_{1}$ at a vertex $\gamma$, there is a linear
map $\phi_{\sigma_{2}, \sigma_{1}}$, which transforms
the variables  $(u_{1}, v_{1})$ attached to $\sigma_{1}$
into
the variables $(u_{2},v_{2})$ attached to $\sigma_{2}$.

With $\gamma=(0,0)$ and $v_{1}=u_{2}$, the transition map 
$\phi_{\sigma_{2}, \sigma_{1}}$ is given by 
\[
	\left[
  	\begin{array}{c}
   	 u_{2}\\ 
    	v_{2}
    \end{array}
	\right] 
	= \left[
  \begin{array}{cc}
    0 & b\\
    1 & a 
	\end{array}
	\right]\,
	\left[
  \begin{array}{c}
    u_{1}\\ 
    v_{1}
  \end{array}
	\right] 
\]
where $a,b\in \RR$ and $b\neq 0$. We choose these constant transition
maps to define the space of splines $\cS^{1}(\cM)$. The gluing
conditions along the edges correspond then to $C^{1}$ conditions for
the polynomials expressed in the same coordinate system.
In this case, $\cS^{1}(\cM)$ is the vector space of piecewise polynomial functions on $\cM$,
which are $C^{1}$ on $\Omega$,
that is, the classical $C^{1}$-spline functions on $\Omega$.

If $a = 0$, the edge of $\sigma_{2}$ at
$\gamma$ distinct from $\tau$ is aligned with the edge of $\sigma_{1}$
at $\gamma$ distinct from $\tau$. The vertex $\gamma$ is a crossing vertex
($\cross_{+}(\gamma)=1$; all the coefficients $a$ in the transition maps
around $\gamma$ vanish) if there are $4$ edges at $\gamma$, which are
pair-wise aligned.

As for any interior edge $\tau\in \cM_{1}$ the transition map is
constant, we have $n_{\tau}=0$, $\mu_{\tau}=0$, $\nu_{\tau}=0$,
$m_{\tau}=1$, $\ms_{\tau}\le 5$ and $d_{\tau}(k)= 2\, k$.
For the boundary edges, we have $d_{\tau}(k)= 2\,k + 2$ (see Remark \ref{rem:boundary:edge}).

We deduce from Theorem \ref{thm:dim} that for $k\geq 5$, we have
\begin{eqnarray*}
  \dim \cS^{1}_{k} & = & {1\over 2} (k-5)(k-4) F_{\Delta}
                         +  2k\, F_{1}^{o} + (2k+2)\,F_{1}^{b} + 3 F_{\Delta}
                       -9\, F_{1} +3\, F_{0} + F_{+}\\
  & = & {1\over 2} (k+2)(k+1)\, F_{\Delta} - 6\,(k-2) \, F_{\Delta} +
        (2k-9)\,  F_{1}^{o} + (2k-7)\, F_{1}^{b} 
                  + 3\, F_{0} + F_{+}
\end{eqnarray*}
where $F_{1}^{o}$ (resp. $F_{1}^{b}$) is the number of interior (resp. boundary) edges.
Using the relations $3 F_{\Delta}= 2 F_{1}^{o} + F_{1}^{b}$
(counting the edges per triangle, we count twice the interior edges
shared by two triangles and once the boundary edges),
$F_{1}^{b} = F^{b}_{0}$ and $F_{0}= F_{0}^{o}+ F_{0}^{b}$
where $F_{0}^{o}$ (resp. $F_{0}^{b}$) is the number of interior
(resp. boundary) vertices, we obtain
\begin{eqnarray*}
	\dim \cS^{1}_{k} & = & {1\over 2} (k+2)(k+1)\, F_{\Delta} - (2 k+1) \, 
	F_{1}^{o}  + 3\, F^{o}_{0} + F_{+}.
\end{eqnarray*}
This coincides with the dimension formula of $C^{1}$
piecewise polynomials of degree $k\geq 5$ on a triangular
planar mesh, given in \cite{ms}. Here $F_{+}$ counts the number of crossing
vertices, also called singular vertices in \cite{ms}. 

The basis functions constructed as in Section \ref{sec:basis} are as follows:
\begin{itemize}
	\item For each vertex $\gamma$, there are 3 basis functions associated to
  	the evaluation and derivatives in $x,y$ at $\gamma$. There are
  	$F(\gamma)-\sum_{\tau\ni \gamma}\cross_{\tau}(\gamma)+ \cross_{+}(\gamma)$ 
  	basis functions associated to the free
 		cross-derivatives on the triangles containing  $\gamma$.
	\item For each interior edge $\tau=(\gamma,\gamma')$, there are $2k-9 
  	+ \cross_{\tau}(\gamma)+\cross_{\tau}(\gamma')$ basis functions
  	associated to $k+1-6 = k-5$ free interior Bernstein coefficients 
  	$b_{3,0},\ldots,b_{k-3,0}$ on the edge, 
  	$k-4$ free interior Bernstein coefficients $b_{2,1},\ldots, b_{k-2,1}$ on one triangle 		$\sigma$ which contains
  	$\tau$ and $b_{2,0}$ (resp. $b_{k-2,0}$) if $\tau$ is a crossing vertex
  	at $\gamma$ (resp. $\gamma'$).
	\item For each boundary edge $\tau'$, there are $2k-7$ basis functions
  	associated to $k-3$ free interior Bernstein coefficients 
  	$b_{2,0},\ldots,b_{k-2,0}$ on the edge $\tau'$, 
		and $k-4$ free interior Bernstein coefficients $b_{2,1},\ldots,
		b_{k-2,1}$ on the triangle $\sigma$ which contains $\tau'$.
	\item For each triangle $\sigma$, there are ${k-4 \choose 2}$ 
		basis functions 	associated to
		the interior Bernstein coefficients $b_{i,j}$ 
		with $2\leq i,j \leq k-2$ and $0\le i+j\le k$.
\end{itemize}
This basis description involves the Bernstein coefficients of
polynomial on the triangles. The basis differs from the nodal basis proposed
in \cite{ms}. From the listed Bernstein coefficients, we can however
recover the nodal basis of \cite{ms}, dual to the evaluation and
derivatives at the vertices and at interior points of the edges and the triangles.

\subsection{A round corner}

We consider a mesh $\cM$ composed of $3$ rectangles $\sigma_{1},\sigma_{2},\sigma_{3}$ glued
around an interior vertex $\gamma$, along the $3$ interior edges $\tau_{1},
\tau_{2}, \tau_{3}$. There are $6$ boundary edges and $6$ boundary vertices.
\begin{figure}[ht]
	\begin{center}
	\includegraphics[width=5.4cm]{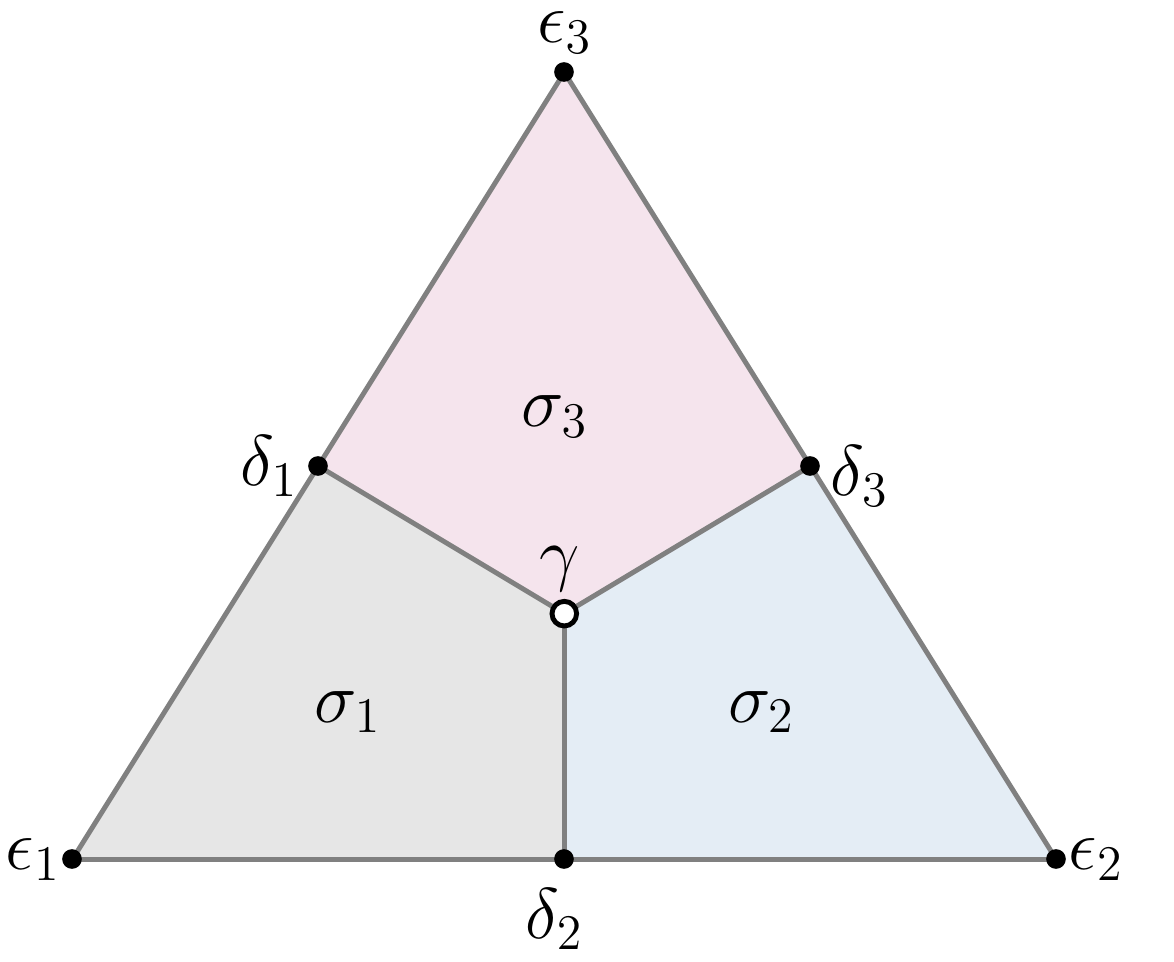}
	\end{center} 
	\vspace{-0.5cm}
	\caption{Smooth corner.}\label{fig:roundcorner}
\end{figure}

We take symmetric gluing data at $\gamma$ and at the crossing boundary vertices $\delta_{i}$.
The transition map across
the interior edge $\tau_{i}$ is
given by the polynomials:
$[a,b,c] = [(u-1), -1, 1 ]$
where $\gamma$ is the end point with $u=0$ and $\delta_{i}$ is the end
point with $u=1$.
The generating syzygies are
\[ 
	S_{1}=[0, 1, 1],
	S_{2}=[1, u, 1].
\]
For the interior edges $\tau_{i}$, we have $n=1$, $m=0$, $\mu=0$, $\nu=1$ and $d_{\tau_{i}}(k)= k+1 + k =
2k+1$.
For the boundary edges $\tau'$, we have $n=0$, $m=0$, $\mu=0$, $\nu=0$
and $d_{\tau'}(k)=2\, (k+1)$.

As $a(0)=-1$ (resp. $a(1)=0$),
$\gamma$ is not a crossing vertex ($\cross_{\tau_{i}}(\gamma)=0$) and $\delta_{i}$
is a crossing vertex of $\tau_{i}$  ($\cross_{\tau_{i}}(\delta_{i})=1$).

We check that the separability of all the interior edges is $4$. For $k=4$, the dimension of $S_{k}^{1}(\cM)$ is
\[ 
	3 \times  (4-3)^{2}  
	+ 3 \times (2 \times 4+1 )  
	+ 6 \times (2 \times 4+2 )  
	+ 4\times 3   
	- 9 \times 9   
	+ 3\times 7   
	+ 6        
	= 48.
\]
The basis functions are constructed as in Section \ref{sec:basis}, using the
algorithms of Appendix A.

\begin{itemize}
	\item The number of basis functions attached to $\gamma$ is $6= 1+2+3$.
	\begin{itemize}
		\item The basis function associated to the value at $\gamma$ is
		\[
			[
  		b_{0,0}+b_{1,0}+b_{0,1}+b_{1,1},
  		b_{0,0}+b_{1,0}+b_{0,1}+b_{1,1},
  		b_{0,0}+b_{1,0}+b_{0,1}+b_{1,1}
  		]
		\]
		\item The two basis functions associated to the derivatives at $\gamma$ are
		\[ 
		\begin{array}{l}
  	\;\bigl[
		\frac{1}{4}\,b_{{1,0}}+\frac{1}{4}\,b_{{1,1}}+{\frac {7}{12}}\,b_{{2,0}}+
		{\frac {7}		{12}}\,b_{{2,1}}+\frac{1}{8}\,b_{{1,2}},\\
		\quad	-\frac{1}{4}\,b_{{0,1}}-\frac{1}{4}\,b_{{1,1}}-\frac{1}{8}\,b_{{2,1}}-{\frac
  	{7}{12}}\,b_{{0,2}}-b_{{1,2}},\\
  	\quad \frac{1}{4}\,b_{{0,1}}-\frac{1}{4}\,b_{{1,0}}-{\frac {7}{12}}\,b_{{2,0}}
  	-{\frac {11}{24}}\,b_{{2,1}}+{\frac {7}{12}}\,b_{{0,2}}
  	+{\frac {7}{8}}\,b_{{1,2}}
  	\bigr]\\
		\;\bigl[
		\frac{1}{4}\,b_{{0,1}}+\frac{1}{4}\,b_{{1,1}}+\frac{1}{8}\,b_{{2,1}}
		+{\frac{7}{12}}\,b_{{0,2}}+b_{{1,2}},\\
		\quad -\frac{1}{4}\,b_{{0,1}}+\frac{1}{4}\,b_{{1,0}}+{\frac {7}{12}}\,
  	b_{{2,0}}+{\frac {11}{24}}\,b_{{2,1}}-{\frac {7}{12}}\,b_{{0,2}}-
  	{\frac {7}{8}}\,b_{{1,2}},\\
  	\quad -\frac{1}{4}\,b_{{1,0}}-\frac{1}{4}\,b_{{1,1}}-{\frac {7}{12}}\,
  	b_{{2,0}}-{\frac {7}{12}}\,b_{{2,1}}-\frac{1}{8}\,b_{{1,2}}
  	\bigr]
		\end{array}
		\]
	\item The three basis functions associated to the cross derivatives at $\gamma$ are 
		\[
		\begin{array}{l}
		\; \bigl[
		\frac{1}{16}\,b_{{1,1}}-\frac{1}{12}\,b_{{2,0}}-\frac{1}{24}\,b_{{2,1}}
		-\frac{1}{12}\,b_{{0,2}}-\frac{1}{8}\,
		b_{{1,2}},\\\quad -\frac{1}{12}\,b_{{2,0}}-\frac{1}{12}\,b_{{2,1}},
		-\frac{1}{12}\,b_{{0,2}}-\frac{1}{6}\,b_{{1,2}}
		\bigr]\\
		\;\bigl[
		-\frac{1}{12}\,b_{{0,2}}-\frac{1}{6}\,b_{{1,2}},\\
		\quad \frac{1}{16}\,b_{{1,1}}-\frac{1}{12}\,b_{{2,0}}-\frac{1}{24}
		\,b_{{2,1}}-\frac{1}{12}\,b_{{0,2}}-\frac{1}{8}\,b_{{1,2}},
		-\frac{1}{12}\,b_{{2,0}}-\frac{1}{12}\,b_{{2,1}}
		\bigr]\\
		\;\bigl[
		-\frac{1}{12}\,b_{{2,0}}-\frac{1}{12}\,b_{{2,1}},\\
		\quad -\frac{1}{12}\,b_{{0,2}}-\frac{1}{6}\,b_{{1,2}},\frac{1}{16}
		\,b_{{1,1}}-\frac{1}{12}\,b_{{2,0}}-\frac{1}{24}\,b_{{2,1}}
		-\frac{1}{12}\,b_{{0,2}}-\frac{1}{8}\,b_{{1,2}}
		\bigr]
		\end{array}  
		\]
	\end{itemize}
 	\item The number of basis functions attached to $\delta_{i}$ is 
 		$4= 1+2+2-1$. Here are the $4$ basis functions associated to $\delta_{1}$: 
		\[
			\begin{array}{l}
			\;[b_{3,0}+b_{3,1}+b_{4,0}+b_{4,1}, 0, b_{0,3}+b_{1,3}+b_{0,4}+b_{1,4} ],\\
			\;[b_{3,0}+b_{3,1}, 0, b_{0,3}+b_{1,3} ],\\
			\;[b_{3,1}+b_{4,1}, 0, -b_{1,3}-b_{1,4} ],\\
  		\;[b_{3,1}, 0,-b_{1,3} ].
			\end{array}
		\]
		The basis functions associated to the other boundary points 
		$\delta_{2},\delta_{3}$ are obtained by cyclic permutation.
	\item The number of basis functions attached to the remaining boundary
  	points is $4= 1+2+1$. For $\epsilon_{1}$, the $4$ basis functions  are 
		\[
		\;[b_{3,3}+b_{3,4}+b_{4,3}+b_{4,4}, 0,0 ],\\
		\;[b_{3,3}+b_{4,3}, 0, 0 ],\\
		\;[b_{3,3}+b_{3,4}, 0, 0 ],\\
		\;[b_{3,3}, 0, 0 ]
		\]
		The basis functions associated to the other boundary points are
		obtained by cyclic permutation.
	\item The number of basis functions attached to the edge $\tau_{i}$
 		is $2\times 4-7=1$. For the edge $\tau_{1}$, it is
		\[ 
			[b_{2,1}, 0, -b_{1,2}].
		\]
		The basis functions associated to the other interior edges are obtained by 
		cyclic permutation.

	\item The number of basis functions attached to the boundary edges is
  	$2 (4-3)=2$. For the boundary edge $(\epsilon_{1},\delta_{1})$ of $\sigma_{1}$, 
  	the two basis functions are
		\[ 
			[ b_{3,2}, 0, 0],  [b_{4,2},0,0].
		\]
	\item The number of basis functions attached to a face $\sigma_{i}$ is 
		$(4-3)^{2}=1$. The basis function associated to $\sigma_{1}$ is
		\[
			[b_{2,2},0,0]
		\]
		and the two other ones are obtained by cyclic permutation.
\end{itemize}

\subsection{A pruned octahedron}
We consider a mesh $\cM$ with 6 triangular faces \fc{AEF}, \fc{CEF},
\fc{ABE}, \fc{BCE}, \fc{ADF}, \fc{CDF}
and one rectangular face \fc{ABCD},
depicted in Figure \ref{fig:aocta} as the {\em Schlegel diagram} of a
convex polyhedron in $\RR^3$.
It is an octahedron where an edge \fc{BD} is removed and
two triangular faces are merged into a rectangular face (see
\cite{VidunasAMS} for the complete octahedron, which involves only
triangular faces).

\begin{figure}[ht]
	\begin{center}
	\includegraphics[width=5cm]{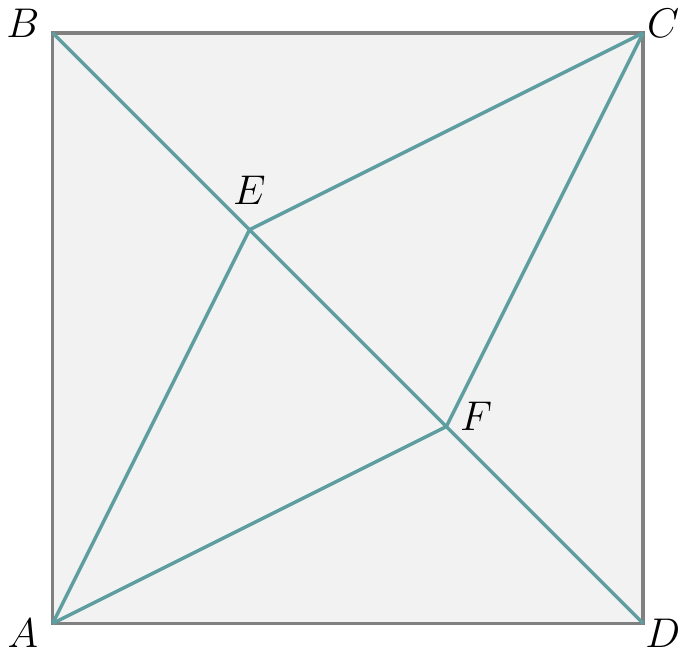}
	\end{center} 
	\vspace{-0.5cm}
	\caption{A pruned octahedron}\label{fig:aocta}
\end{figure}

We are going to use the following notation for the variables
on the faces of $\cM$. For $X,Y\in \RR^{2}$ two vertices defining an
edge \fc{XY} of a face $\sigma$, let 
$u_{X}^{Y}:\RR^2\to\RR$ be the linear function with
$u_{X}^Y(X)=0$,
$u_{X}^Y(Y)=1$ and $u_{X}^{Y}(Z)=0$ for all the points $Z$ on the 
edge of $\sigma$ through $X$ and distinct from $XY$.
We will use these linear functions $u_E^B$, $u_F^A$, etc., as
variables on the different faces.
As the restriction on a share edge \fc{XY} of the two functions
defined on the faces adjacent to \fc{XY} coincide, there is no
ambiguity in evaluating these linear functions on $\cM$.
For a triangular face \fc{XYZ}, we have 
$u_{X}^{Y}+u_{Y}^{Z}+u_{Z}^{X}=1$.
For a rectangular face \fc{XYZW}, we have 
$u_{X}^{Y}=u_{W}^{Z}$ and $u_{X}^{W}=u_{Y}^{Z}$.
We denote by $\partial_{XY}$ the derivative with respect to the
variable $u_{X}^{Y}$. It is such that $\partial_{XY}(u_X^Y)=1$. 
On a triangular face \fc{XYZ}, we have 
$\partial_{XY}+\partial_{YZ}+\partial_{ZX}=0$.

We use a symmetric gluing at all the vertices and therefore the
vertices $A,C,E,F$ are crossing vertices. Let us describe how we 
construct the gluing data on the edges by interpolation at the
vertices, in a smaller degree than the degree associated to the gluing
\eqref{eq:hahngluing} proposed in \cite{Hahn87}.

In terms of differentials (see relation \eqref{eq:differentialrelation}), the symmetric gluing at the vertices translates as 
\begin{align} \label{eq:horder4}
	\partial_{EA}+\partial_{EC}=0, \qquad \partial_{EB}+\partial_{EF}=0 
	\qquad \mathrm{at}\ 	E,\nonumber \\
	\partial_{FA}+\partial_{FC}=0, \qquad \partial_{FD}+\partial_{FE}=0 
	\qquad \mathrm{at}\ F,\nonumber \\
	\partial_{AB}+\partial_{AF}=0, \qquad \partial_{AE}+\partial_{AD}=0 
	\qquad \mathrm{at}\ A,\\
	\partial_{CB}+\partial_{CF}=0, \qquad \partial_{CE}+\partial_{CD}=0
	\qquad \mathrm{at}\ C, \nonumber
\end{align}
and around the vertices of order 3:
\begin{equation*} 
	\partial_{BA}+\partial_{BC}+\partial_{BE}=0 \quad \mathrm{at}\ B, \qquad 
	\partial_{DA}+\partial_{DC}+\partial_{DF}=0 \quad \mathrm{at}\ D.
\end{equation*}
For gluing the triangles \fc{EFA} and \fc{EFC} along $EF$,
we interpolate the following relations between the derivatives:
\begin{equation*}
	\partial_{EA}+\partial_{EC}=0 \ \mathrm{at}\ E,
	\quad \partial_{EA}+\partial_{EC}=2\,\partial_{EF}  \ \mathrm{at}\ F,
\end{equation*}
where the second expression is $\partial_{FA}+\partial_{FC}=0$ rewritten using
$\partial_{FA}=\partial_{EA}-\partial_{EF},
\partial_{FC}=\partial_{EC}-\partial_{EF}$.
We choose the linear interpolation
\begin{equation*}
	\partial_{EA}+\partial_{EC}=2u_E^F\,\partial_{EF}.
\end{equation*}
Thereby we have $\ma_{EF}=2u_E^F$, $\mb_{EF}=-1$ and the gluing data
for the edge $EF$ is $[2\,u_E^F, -1, 1]$. 

For the edge \fc{EB} between the triangles \fc{EBA} and \fc{EBC},
we interpolate the following relations:
\[
	\partial_{EA}+\partial_{EC}=0 \ \mathrm{at}\ E,
	\qquad \partial_{EA}+\partial_{EC}=3\,\partial_{EB}\ \mathrm{at}\ B,
\]
where the latter relation is $\partial_{BA}+\partial_{BC}+\partial_{BE}=0$ rewritten using
$\partial_{BA}=\partial_{EA}-\partial_{EB}, 
\partial_{BC}=\partial_{EC}-\partial_{EB}, 
\partial_{BE}=-\partial_{EB}$. 
Additionally, we have to take into account the compatibility conditions
\eqref{eq:ddv1}-\eqref{eq:ddv2} since $E$ is a crossing vertex. It translates as 
$\partial_{EB} (\ma_{EB})=\partial_{EF}(\ma_{EF})$
and $\partial_{EC} (\ma_{EC})=\partial_{EA}(\ma_{EA})$ at the vertex $E$.
This leads to the following gluing data on the edge \fc{EB}:
\[
	\fc{EB}: [ 2u_E^B+(u_E^B)^2, -1,1].
\]
Similarly, the gluing data of the edge \fc{FD} is
\[
	\fc{FD}: [2u_F^D+(u_F^D)^2, -1,1].
\]
The edges $EA$, $EC$, $FA$, $FC$ connect cross vertices just as $EF$
and yield linear gluing data
\begin{align*}
	\fc{EF}: \quad [2u_E^A,  -1,1],\qquad
	\fc{EC}: \quad [2u_E^C, -1,1],\\
	\fc{FA}: \quad [2u_F^A, -1,1],\qquad
	\fc{FC}: \quad [2u_F^C, -1,1]
\end{align*}
We check that the compatibility conditions
\eqref{eq:ddv1}-\eqref{eq:ddv2} are satisfied across $EA,EC$ and $FA,FC$.
The gluing data along $AB$, $AD$, $CB$, $CD$ looks the same:
\begin{align*} 
	\fc{AB}: \quad[2u_A^B,-1,1], \qquad
	\fc{AD}: \quad[2u_A^D,-1,1], \\
	\fc{CB}: \quad[2u_C^B,-1,1], \qquad
	\fc{CD}: \quad[2u_C^D,-1,1].
\end{align*}
We have linear gluing data everywhere except on the edges \fc{EB} and
\fc{FD}. Let us analyze the syzygies associated this data.
\begin{itemize}
	\item For the edges \fc{EB} and
		\fc{FD} with one crossing vertex, the gluing data is of
		the form $[2 u + u^{2},-1,1]$.
		We have $n=2$ and $m=1$ since the edge is connecting two triangles,
		$\mu=0$ and $\nu=2$ and $d(k)=2 k-2$. 
		The $\mu$-basis is $[0,1,1]$, $[-1, -2u-u^{2},0]$.
		The separability is achieved in degree $k\ge 6$ and not $5$ as it
		could be expected ($d(5)\ge 8$).
 	\item For the edges \fc{EA}, \fc{EA}, \fc{FA}, \fc{FC}, \fc{EF} between 
		triangular faces, with two crossing vertices,
		the linear gluing data is of the form $[2u,1,-1]$.
		We have $n=1$ and $m=1$, $\mu=0$, $\nu=1$ and $d(k)=2 k-1$. 
		The $\mu$-basis is $[0,1,1]$, $[-1, -2u,0]$.
		The separability is achieved in degree $k\ge 4$.
	\item For the edges \fc{AB}, \fc{AD}, \fc{CB}, \fc{CD} between 
		a triangular face and a rectangular face, with one crossing vertices,
		the linear gluing data is of the form $[2u,1,-1]$.
		We have $n=1$ and $m=0$, $\mu=1$ since  the degree of the
		homogeneization $[d_{a},d_{b},d_{c}]$ (see Definition \ref{def:Fm}) is
		$[3,2,1]$ or $[3,1,2]$, $\nu=1$ and $d(k)=2 k$. 
		The $\mu$-basis is $[0,1,1]$, $[-1, -2u,0]$.
		The separability is also achieved in degree $k\ge 4$.
\end{itemize}
Now we count how many splines do we have in degree $k\ge 6$:
\begin{itemize}
	\item For the four crossing vertices $A,C,E,F$ we have 
		$1+2+1=4$ dimensions and $1+2+3=6$ dimensions for $B$ and for $D$.
		In total we have $4\cdot4+2\cdot 6=28$ degrees of freedom around the
		vertices of $\cM$.
	\item For the edges \fc{EB} and \fc{FD}, we have $2(k-2)-8=2k-12$ dimensions.
  	For the edges \fc{EA}, \fc{EA}, \fc{FA}, \fc{FC}, \fc{EF}, we have
  	$2k -1- 7=2k-8$ dimensions.
  	For the edges \fc{AB}, \fc{AD}, \fc{CB}, \fc{CD}, we have
  	$2k-8$.
	\item For the 6 triangular faces, we have $\,{k-4\choose 2}$
   	dimensions and for the rectangular face $(k-3)^{2}$.
\end{itemize}  
The dimension formula in degree $k\ge 6$ is then
\begin{equation*}
	28-4+11\cdot2\,(k-4)+6\,{k-4\choose 2}+(k-3)^2 =(2k-3)^2+k-4.
\end{equation*}
For $k=6$, the dimension is $83$.
It turns out that this formula also holds for degree $k=4$, $k=5$.

The construction of basis functions can be done as described in
Section \ref{sec:basis}.
Let us give the basis functions associated to the value and first
derivatives at the point $A$.
Here are the Bernstein coefficients in degree $4$ of the basis function for the value at $A$
with the vertex A represented in the center and the edges represented by horizontal and vertical central lines (in bold): 
\begin{align*}
	\begin{array}{ccccccc}
		&&& \bf 0 \\[-5pt]
		& \iddots\!\! & 0 & \bf \frac12 & 0 & \!\!\ddots \\[2pt]
		& 0 & 1 & \bf 1 & 1 & 0 \\[2pt]
		\bf 0& \bf 0 & \bf 1 & \bf 1 & \bf 1 & \bf \frac12 & \bf 0 \\[2pt]
		0 & \!\!-1\!\! & 1 & \bf 1 & 1 & 0 \\[-5pt]
		\vdots & 0 & \!\!-1\!\! & \bf 0 & 0 & \!\!\iddots  \\[1pt]
		&  \!\!\cdots\!\!  &0 & \bf 0 \\
	\end{array} \qquad
	\begin{picture}(130,20)(-30,-4)
		\put(0,0){\vector(1,0){100}}  \put(0,0){\vector(-1,0){0}}
		\put(50,-50){\vector(0,1){100}}   \put(50,-50){\vector(0,-1){0}}
		\put(52,4){$A$}  \put(92,5){$F$} \put(53,42){$E$}  
	\put(0,5){$B$} \put(53,-50){$D$}  
\end{picture}
\end{align*}
This gives the following specializations to the polygons \fc{ABE}, \fc{AEF}, \fc{AFD}, \fc{ABCD} (respectively), selectively de-homogenized:
\begin{align*}
	& (u_{E}^A)^2 \, (1+3u_{A}^B)(1+2u_{A}^E-u_{A}^B), \\
	& (u_{E}^A)^2 \, (1+2u_{A}^F+2u_{A}^E+6u_{A}^Fu_{A}^E),\\
	& (u_{F}^A)^2 \, (1+3u_{A}^D)(1+2u_{A}^F-u_{A}^D), \\
	& (u_{B}^Au_{D}^A)^2 \left(u_{B}^Au_{D}^A(1+3u_{A}^D)(1+3u_{A}^B)
	-24u_{A}^Du_{A}^B(u_{AB}^Du_{B}^A+u_{A}^Bu_{D}^A)\right)\!,
\end{align*}
and $0$ on the other faces.
The basis function associated to the first derivative in one of the
directions at the cross vertices is:
\begin{align} \label{eq:devspline}
	\begin{array}{ccccccc}
		&&& \bf 0 \\[-5pt]
		& \iddots\!\! & 0 & \bf \frac16 & 0 & \!\!\ddots \\[2pt]
		& 0 & \frac13 & \bf \frac14 & \frac13 & 0 \\[2pt]
		\bf 0& \bf 0 & \bf 0 & \bf 0 & \bf 0 & \bf 0 & \bf 0 \\[2pt]
		0 & \frac{7}{24} & \!\!-\frac5{16}\! & \bf \!\!-\frac14\! & \!\!-\frac13\! 
		& 0 \\[-5pt]
		\vdots & 0 & \!\!-\frac16\!\! & \bf 0 & 0 & \!\!\iddots  \\[1pt]
		&  \!\!\cdots\!\!  &0 & \bf 0 \\
	\end{array} 
\end{align}
The non-zero specializations to  \fc{ABE}, \fc{AEF}, \fc{AFD}, \fc{ABCD}  are, respectively:
\begin{align*}
	& (u_{E}^A)^2 \,u_{A}^E\, (1+3u_{A}^B), \  
	(u_{E}^A)^2 \,u_{A}^E\, (1+3u_{A}^F), \,
	-(u_{F}^A)^2 \,u_{A}^D\,  (u_{F}^A+4u_{A}^F), \\
	& (u_{B}^Au_{D}^A)^2 \,u_{A}^D \left(-u_{D}^A(1+3u_{A}^B)+7u_{A}^Bu_{B}^Au_{A}^D
	\right).
\end{align*}
The basis function for the derivative in the other direction is obtained by a mirror image of (\ref{eq:devspline}).

The basis function corresponding to the cross derivatives is realized by
\begin{align*}
	\begin{array}{ccccccc}
		&&& \bf 0 \\[-5pt]
		& \iddots\!\! & 0 & \bf 0 & 0 & \!\!\ddots \\[2pt]
		& 0 & \!\!-\frac1{12}\!\! & \bf 0 & \frac1{12} & 0 \\[2pt]
		\bf 0& \bf 0 & \bf 0 & \bf 0 & \bf 0 & \bf 0 & \bf 0 \\[2pt]
		0 & \frac{1}{24} & \frac1{16} & \bf 0 & \!\!-\frac1{12}\!\! & 0 \\[-5pt]
		\vdots & 0 & \frac1{24} & \bf 0 & 0 & \!\!\iddots  \\[1pt]
		&  \!\!\cdots\!\!  &0 & \bf 0 \\
	\end{array} 
\end{align*}
In this spline, we could modify the 0 entry next to two $\frac1{24}$ entries
to $\frac1{36}$, so to lower the degree of the specialization to the rectangle.
After the modification, the 4 non-zero specializations would be
\begin{align*}
	-(u_{E}^A)^2 \,u_{A}^E\, u_{A}^B, \ (u_{E}^A)^2 \,u_{A}^E\, u_{A}^F, \,
	-(u_{F}^A)^2 \,u_{A}^D\, u_{A}^F,  \, (u_{B}^Au_{D}^A)^2 \,u_{A}^D \, u_{A}^B.
\end{align*}
The local splines around $C$ look the same.
The local splines around other vertices involve the edges $EB$ and $FD$,
and we would need degree 6 splines.

\appendix

\section{Algorithms for the basis construction}

Our input data is the topological surface $\cM$ and the gluing data.
For each edge $\tau$ of $\cM$, we are given 
the $\mu_{\tau}$-basis of $Z_{k}$:
$$
	S_{1}^{\tau} =[A_{1}^{\tau}, B_{1}^{\tau},C_{1}^{\tau}],\ \ 
	S_{2}^{\tau} =[A_{2}^{\tau}, B_{2}^{\tau},C_{2}^{\tau}].
$$
The rational map is then described by 
$
	\ma_{\tau} = \frac{a_{\tau}}{c_{\tau}}, \mb_{\tau} = \frac{b_{\tau}}{c_{\tau}}
$
with 
$a_{\tau}=B_{1}^{\tau} C_{2}^{\tau}-B_{2}^{\tau} C_{1}^{\tau}$,
$b_{\tau}=A_{2}^{\tau} C_{1}^{\tau}-A_{1}^{\tau} C_{2}^{\tau}$, and 
$c_{\tau}=A_{1}^{\tau} B_{2}^{\tau}-A_{2}^{\tau} B_{1}^{\tau}$.

The spline basis functions $f=(f_{\sigma})$ are
represented on each face $\sigma$ by their coefficients in the Bernstein basis
of the face in degree $k$:
\[
	f_{i} = \sum_{l,m} c_{l,m}^{i} b_{l,m}^{\sigma}(u_{i},v_{i})
\]
Let $e_{\sigma}(k)=k^{2}$ if $\sigma$ is a rectangular face and $e_{\sigma}(k)=k\, (k-1)$ if $\sigma$ is a triangular face.

\subsection{Vertex basis functions}
Let $\gamma$ be a vertex of $\cM$ shared by the faces
$\sigma_{1},\ldots, \sigma_{F} $ and such that $\sigma_{i}$ and $\sigma_{i+1}$ share the edge
$\tau_{i+1}$.
We compute the Bernstein coefficients
$\cb^{i}=[c_{0,0}^{i}, c_{1,0}^{i}, c_{0,1}^{i}, c_{1,1}^{i},\ldots]$
of the basis functions attached to a vertex $\sigma_{i}$, using the
equations \eqref{eq:recg1.1}, \eqref{eq:recg1.2} and the relation between the Bernstein
coefficients and the Taylor coefficients of the function at $(0,0)$, see \eqref{eq:TaylorBernsteinTriang}, \eqref{eq:TaylorBernsteinQuad}.

If $c_{0,0}^{i}$ corresponds to the point $\gamma$, with 
coordinates $(0,0)$ in the face $\sigma_{i}$
and $f_{\sigma_{i}} = p + q_{i} u_{i} + q_{i+1} v_{i} +
s_{i} u_{i}v_{i}+ \cdots$ are $\cb^{i}$, we use the relations
$p=c_{0,0}^{i},
q_{i}=k\, c_{1,0}^{i}, \;
q_{i+1}=k\, c_{0,1}^{i},\;
s_{i}=e_{k}\,(c_{1,1}^{i}-c_{1,0}^{i}- c_{0,1}^{i}).$

\ \\
\begin{algorithm}[H] 
	{\textsc{Basis function for the value at vertex $\gamma$}}
 
 	\bigskip

 	\For{i in [1,F]}{let $c_{0,0}^{i}:=1,c_{1,0}^{i}:=1,c_{0,1}^{i}:=1,c_{0,1}^{i}:=1$ and
   $c_{l,m}^{i}:=0$
   for $(l,m)\not\in\{ (0,0),(1,0),(0,1),(1,1)\}$\; }
\end{algorithm}

 
\begin{algorithm}[H]  
	\textsc{Basis functions for the derivatives at vertex $\gamma$}
 
 	\bigskip
 
 	\For{$[c_{0,0}^{1},c_{0,1}^{1}]$ in $\{[1,0], [0,1]\}$} {
 	\For{i in [2,F]}{
	\[ 
	\left[
	\begin{array}{c}
	c_{1,0}^{i}\\ 
	c_{0,1}^{i}
	\end{array}
	\right]
	= \left(
	\begin{array}{cc}
	0 & 1 \\
	\mb_{\tau_{i}}(0) & \ma_{\tau_{i}}(0) \\
	\end{array}
	\right)
	\left[
	\begin{array}{c}
	c_{1,0}^{i-1}\\ 
	c_{0,1}^{i-1}
	\end{array}
	\right]\ 
	\]}
	\If{all edges $\tau_{i}$ are crossing edges at $\gamma$}
	{let $c_{1,1}^{1}:=0$\;}
	\For{i in [2,F]}{
 	\eIf{$\tau_{i}$ is a crossing edge at $\gamma$}{
	$\begin{array}{rcl}
 		c_{1,1}^{i} &= & c_{1,0}^{i}+c_{0,1}^{i} \\
		&&+ \frac{1}{e_{\sigma_{i}}(k)}\left(e_{\sigma_{i-1}}(k) \mb_{\tau_{i}}(0)
		\, (c_{1,1}^{i-1}-c_{1,0}^{i-1}-c_{0,1}^{i-1}) \right.\\&&\left.+ k
		\, \ma_{\tau_{i}}'(0)\,
		c_{0,1}^{i-1} + k\, \mb_{\tau_{i}}'(0)\, c_{1,0}^{i-1} \right)
	\end{array}$
 	}{$c_{1,1}^{i}:= 0$\;}
	}
	\For{i in [1,F']}{lift($\cb_{i-1},\cb_{i}$, $\tau_{i}$)}
	}
\end{algorithm}


\begin{algorithm}[H] 
	\textsc{Basis functions for the cross derivatives around $\gamma$}


	\For{i in [1,F]}{let $c_{0,0}^{i}:=0, c_{1,0}^{i}:=0,
 	c_{0,1}^{i}:=0$\;}
	\eIf{all edges $\tau_{i}$ are crossing edges at $\gamma$}
	{ let $c_{1,1}^{1}:=1$\;
	\For{i in [2,F]}{
 	$c_{1,1}^{i} = \frac{e_{\sigma_{i-1}}}{e_{\sigma_{i}}(k)} \mb_{\tau_{i}}(0)
 	\, c_{1,1}^{i-1}$\;
 	}
	\For{i in [1,L]}{lift($\cb_{i},\cb_{i-1}$, $\tau_{i}$)\;}
	}
	{
	\For{j in [1,F] such that $\tau_{j}$ is not a crossing edge at
 	$\gamma$}
	{
  let $c_{1,1}^{j}=1$ and $c_{1,1}^{l}=0$ for $l\neq j$\;
  \For{i in [1,F']}{
    \If{$\tau_{i}$ is a crossing edge}{$c_{1,1}^{i} = \frac{e_{\sigma_{i-1}}}			{e_{\sigma_{i}}(k)} \mb_{\tau_{i}}(0)\,  c_{1,1}^{i-1}$}
  }
 \For{i in [1,L]}{lift($\cb_{i},\cb_{i-1}$, $\tau_{i}$)\;}
	}
	}
\end{algorithm}


\noindent The function lift($\cb_{i},\cb_{i-1}$, $\tau_{i}$) used in the
algorithm consists in computing the coefficient of a spline function
with support along the edge $\tau_{i}$,  
from its first Taylor coefficients on the
faces $\sigma_{i-1}$, $\sigma_{i}$.


\begin{algorithm}[H]  
	\textsc{lift($\cb_{i},\cb_{i-1}$, $\tau_{i}$)}


	\For{i in [1,F]}{
	solve the systems:
	\[ 
	\left[
	\begin{array}{c}
		k\, c_{1,0}^{i}\\ 
		k\, c_{1,0}^{i-1}
	\end{array}
	\right]
	=
	\left[
	\begin{array}{cc}
		A_{1}(0) & A_{2}(0)\\
 		B_{1}(0) & B_{2}(0)
	\end{array}
	\right]
	\left[
	\begin{array}{c}
 		p_{0}^{i}\\ 
 		q_{0}^{i}
	\end{array}
 	\right]
	\]
	and
	\begin{eqnarray*}
	\lefteqn{
	\hspace{-2cm} \left[
	\begin{array}{c}
		e_{\sigma_{i-1}}(k)\, (c_{1,1}^{i-1} -c_{1,0}^{i-1} -c_{0,1}^{i-1}) \\ 
		-e_{\sigma_{i}}(k) \, (c_{1,1}^{i} -c_{1,0}^{i} -c_{0,1}^{i})
	\end{array}
	\right]
	-
	\left[
	\begin{array}{cc}
 		B'_{1}(0) & B'_{2}(0)\\ 
  	C'_{1}(0) & C'_{2}(0) 
	\end{array}
	\right]
	\left[
	\begin{array}{c}
		 p_{0}^{i}\\ 
 		q_{0}^{i}
	\end{array}
 	\right]}\\ 
	&=&
	\left[
	\begin{array}{cc}
 		B_{1}(0) & B_{2}(0)\\ 
 		C_{1}(0) & C_{2}(0) 
	\end{array}
	\right]
	\left[
	\begin{array}{c}
 		p_{1}^{i}\\ 
 		q_{1}^{i}
	\end{array}
	\right].
	\end{eqnarray*}
	compute ${P}^{i}:=p_{0}^{i} (1-3 u_{i}^{2}+2u_{i}^{3})+
	p_{1}^{i} (u-2 	u_{i}^{2}+u_{i}^{2})$,
	${Q}^{i}:=q_{0}^{i} (1-3 u_{i}^{2}+2u_{i}^{3})+q_{1}^{i} (u-2 u_{i}^{2}+u_{i}^{2})$\;
	compute the image $(g_{i},\tilde{g}_{i})$ of ${P}^{i} S_{1}^{i} + {Q}^{i} S_{2}^{i}$
	by $\Theta_{\tau_{i}}$ and update the coefficients of
	$\cb_{i-1}, \cb_{i}$\;
	}
\end{algorithm}

\bigskip

\noindent As $A_{1}(0) B_{2}(0)-A_{2}(0) B_{1}(0)=c(0)\neq 0$, the first system
has a unique solution.
When $B_{1}(0) C_{2}(0)-B_{2}(0) C_{1}(0)=a(0)\neq 0$ (i.e. when
$\tau_{i}$ is not a crossing edge at $\gamma$), the second system
has a unique solution.
When $a(0)=0$ (i.e. when $\tau_{i}$ is a crossing edge at $\gamma$), the second system is
degenerate, but it still has a (least square) solution.

The polynomials ${P}^{i}$ (resp. $ {Q}^{i}$) are constructed so that
$P^{i}(0)=p_{0}^{i}, P^{i'}(0)=p_{1}^{i}$, $P^{i}(1)=0, P^{i'}(1)=0$
(resp. $Q^{i}(0)=q_{0}^{i}, Q^{i'}(0)=q_{1}^{i}$, $Q^{i}(1)=0, Q^{i'}(1)=0$).

By construction, the Taylor expansions of their image by $\Theta_{\tau_{i}}$ vanish at
$\gamma'$ and coincide with
$[c_{0,0}^{i-1}, c_{1,0}^{i-1},c_{0,1}^{i-1}, e_{\sigma_{i-1}}(k)
(c_{1,1}^{i-1}-c_{1,0}^{i-1}-c_{0,1}^{i-1} ) ]$,
$[c_{0,0}^{i}, c_{1,0}^{i},c_{0,1}^{i}$, $e_{\sigma_{i}}(k) (c_{1,1}^{i}-c_{1,0}^{i}-c_{0,1}^{i} ) ]$
at $\gamma$ respectively on $\sigma_{i-1}$ and $\sigma_{i}$.

\subsection{Edge basis functions}\ \\

\begin{algorithm}[H] 
	\textsc{Basis functions for the edge $\tau$}
 
	\bigskip 
 
	\KwIn{ $[A^{\tau}_{1},B^{\tau}_{1},C^{\tau}_{1}]$,
 	$[A^{\tau}_{2},B^{\tau}_{2},C^{\tau}_{2}]$ the $\mu$-basis of the
 	syzygy module $Z(\tau)$\;}
 
 	\If{$\cross_{\tau}(\gamma) = 1$ }{
	compute the image by $\Theta_{\tau}$ of $u (1-u)^{2} \left(C^{\tau}_{2}(0)\,
 	[A^{\tau}_{1},B^{\tau}_{1},C^{\tau}_{1}] - C^{\tau}_{1}(0)
 	\, [A^{\tau}_{2},B^{\tau}_{2},C^{\tau}_{2}]\right)$
	}

 	\If{$\cross_{\tau}(\gamma') = 1$}{
	compute the image by $\Theta_{\tau}$ of $u^{2} (1-u) \left(C^{\tau}_{2}(1)\,
 	[A^{\tau}_{1},B^{\tau}_{1},C^{\tau}_{1}] - C^{\tau}_{1}(1)
 	\, [A^{\tau}_{2},B^{\tau}_{2},C^{\tau}_{2}]\right)$
 	}
 
 	Let $\Delta= u^{2}(1-u)^{2}$\;
 	\For{i in [0,$k-\mu-m-4$]}{compute the image by $\Theta_{\tau}$ of $u^{i}\Delta 
 	\, [A^{\tau}_{1},B^{\tau}_{1},C^{\tau}_{1}]$.}
 	\For{i in [0,$k-\nu-m-4$]}{compute the image by $\Theta_{\tau}$ of $u^{i}\Delta 
 	\, [A^{\tau}_{2},B^{\tau}_{2},C^{\tau}_{2}]$.}
\end{algorithm}

\subsection{Face basis functions}\ \\

\begin{algorithm}[H] 
	\textsc{Basis functions for the face $\sigma$}
 
	\bigskip 
 
 	\For{$2 \le i \le k-2$, $2\le j \le k-2$ (and $i+j\le k-2$ if $\sigma$
 	is a triangle)}{let $c_{i,j}:=1$ and $c_{i',j'}=0$ for $i'\neq i$
 	or $j'\neq j$.}
\end{algorithm}






\end{document}